\documentclass[12pt,reqno]{amsart}
\usepackage{amssymb,mathrsfs,longtable}

\newcommand{\B}{\mathcal{B}}

\newcommand{\F}{\mathcal{F}}
\newcommand{\cL}{\mathcal{L}}
\newcommand{\M}{\mathcal{M}}
\newcommand{\K}{\mathcal{K}}
\newcommand{\cS}{\mathcal{S}}
\newcommand{\cH}{\mathcal{H}}

\newcommand{\cN}{\mathcal{N}}
\newcommand{\T}{\mathcal{T}}
\newcommand{\Te}{\Theta}
\newcommand{\sT}{\mathscr{T}}
\newcommand{\sL}{\mathscr{S}}

\newcommand{\N}{\mathbb{N}}
\newcommand{\R}{\mathbb{R}}
\newcommand{\Z}{\mathbb{Z}}

\newcommand{\al}{\alpha}

\newcommand{\ga}{\gamma}
\newcommand{\de}{\delta}
\newcommand{\e}{\varepsilon}
\newcommand{\fy}{\varphi}

\newcommand{\la}{\lambda}
\newcommand{\La}{\Lambda}
\newcommand{\te}{\theta}
\newcommand{\s}{\sigma}
\newcommand{\ta}{\tau}
\newcommand{\ka}{\kappa}
\newcommand{\x}{\xi}
\newcommand{\y}{\eta}
\newcommand{\z}{\zeta}

\newcommand{\ro}{\rho}

\newcommand{\De}{\Delta}

\newcommand{\Ga}{\Gamma}

\newcommand{\p}{\partial}
\newcommand{\na}{\nabla}
\newcommand{\re}{\mathop{\mathrm{Re}}}

\newcommand{\supp}{\operatorname{supp}}
\newcommand{\Cu}{\bigcup}

\newcommand{\lec}{\lesssim}
\newcommand{\gec}{\gtrsim}

\newcommand{\I}{\infty}
\newcommand{\da}{\dagger}

\newcommand{\ti}{\widetilde}

\newcommand{\U}{\underline}
\newcommand{\LR}[1]{{\langle #1 \rangle}}

\newcommand{\EQ}[1]{\begin{equation}\begin{split} #1 \end{split}\end{equation}}
\newcommand{\EQN}[1]{\begin{equation*}\begin{split} #1 \end{split}\end{equation*}}
\newcommand{\CAS}[1]{\begin{cases} #1 \end{cases}}
\newcommand{\mat}[1]{\begin{pmatrix} #1 \end{pmatrix}}
\newcommand{\pt}{&}
\newcommand{\pr}{\\ &}
\newcommand{\pq}{\quad}
\newcommand{\pn}{}
\newcommand{\prq}{\\ &\quad}

\numberwithin{equation}{section}

\newtheorem{thm}{Theorem}[section]
\newtheorem{cor}[thm]{Corollary}
\newtheorem{lem}[thm]{Lemma}
\newtheorem{prop}[thm]{Proposition}

\theoremstyle{remark}
\newtheorem{rem}{Remark}[section]

\newcommand{\Soliton}{\operatorname{S}_{\operatorname{oliton}}}
\newcommand{\Static}{\operatorname{S}_{\operatorname{tatic}}}
\newcommand{\Sol}{\operatorname{S}_{\operatorname{olution}}}

\newcommand{\diff}[1]{{\triangleleft #1}}
\newcommand{\pa}{\triangleright}
\newcommand{\pot}{\mathcal{W}}

\newcommand{\sign}{\operatorname{sign}}

\newcommand{\xt}{^{\operatorname{\bf x}}}
\newcommand{\dt}{^{\operatorname{\bf d}}}

\newcommand{\dist}{\operatorname{dist}}

\newcommand{\St}{{\operatorname{St}}}
\newcommand{\Ra}{\mathscr{R}}
\newcommand{\dR}{d_{\Ra}}
\newcommand{\ton}{\text{ on }}
\newcommand{\cmpl}{\complement}

\newcommand{\ra}{\varsigma}

\newcommand{\cHR}[1]{\cH\downharpoonright#1}
\newcommand{\tHR}[1]{{\cH\!\downharpoonright\!#1}}

\newcommand{\dom}{\mathscr{D}}
\newcommand{\ck}{\check}
\newcommand{\Lim}{\lim\limits}
\newcommand{\Limsup}{\varlimsup\limits}
\newcommand{\Liminf}{\varliminf\limits}
\newcommand{\sB}{\mathscr{B}}
\newcommand{\ig}{\iota}

\begin{document}

\title[Center-stable manifold for the critical wave equation]{Center-stable manifold of the ground state in the energy space for the critical wave equation}

\author{Joachim Krieger}
\address{B\^{a}timent des Math\'ematiques, EPFL, Station 8, CH-1015 Lausanne, Switzerland}
\thanks{Support by the Swiss National Science foundation for the first author is acknowledged.}

\author{Kenji Nakanishi} 
\address{Department of Mathematics, Kyoto University, Kyoto 606-8502, Japan}

\author{Wilhelm Schlag}
\address{Department of Mathematics, The University of Chicago, 5734 South University Avenue, Chicago, IL 60615, U.S.A.}
\thanks{Support of the National Science Foundation   DMS-0617854, DMS-1160817 for the third author is gratefully acknowledged.}



\begin{abstract}
We construct a center-stable manifold of the ground state solitons in the energy space for the critical wave equation without imposing any symmetry, as the dynamical threshold between scattering and blow-up, and also as a collection of solutions which stay close to the ground states. Up to energy slightly above the ground state, this completes the 9-set classification of the global dynamics in our previous paper \cite{CNW-nonrad}. We can also extend the manifold to arbitrary energy size by adding large radiation. The manifold contains all the solutions scattering to the ground state solitons, and also some of those blowing up in finite time by concentration of the ground states. 
\end{abstract}

\maketitle 

\tableofcontents

\section{Introduction}
We study global dynamics of the critical wave equation (CW) 
\EQ{ \label{eqCW}
 \pt \ddot u - \De u = f'(u) := |u|^{2^*-2}u, \pq 2^*:=\frac{2d}{d-2}, 
 \pr u(t,x):I\times\R^d\to\R,\pq I\subset\R, \pq d= 3 \ or \  5,}
in the energy space\footnote{The exclusion of $d=4$ is by the same reason as in \cite{CNW-nonrad}, namely to preclude type-II blow-up in the scattering region by \cite{DKM2}. The argument in this paper or \cite{CNW-nonrad} is not sensitive to the dimension.}
\EQ{ \label{def H}
 \vec u(t):=\mat{ u(t) \\ \dot u(t)} \in \cH:=\dot H^1(\R^d)\times L^2(\R^d).}
Henceforth, the arrow on a function $\vec u(t)$ indicates the vector $\vec u(t)=(u(t),\dot u(t))$ given by a scalar function $u(t)$. We do not distinguish column and row vectors. 
The above equation (CW) is in the Hamiltonian form
\EQ{ \label{eq Ham}
 \vec u_t = JE'(\vec u),\pq J:=\mat{0 & 1 \\ -1 & 0},}
and so the energy or the Hamiltonian
\EQ{ \label{def En}
 \pt E(\vec u(t)) := \int_{\R^d} \frac{|\dot u|^2+|\na u|^2}{2}-f(u)dx, \pq f(u):=\frac{|u|^{2^*}}{2^*},}
is conserved. Another important conserved quantity is the total momentum
\EQ{ \label{def P}
 P(\vec u(t)):=\int_{\R^d} \dot u \na u dx.}
The Nehari functional
\EQ{
 K(u):=\int_{\R^d}|\na u|^2-|u|^{2^*}dx}
plays a crucial role in the variational argument. 
(CW) is invariant under translation, the Lorentz transform, and the scaling 
\EQ{
 u(t,x) \mapsto u_\la(t,x):=\la^{d/2-1}u(\la t,\la x),}
which also preserves the energy $E(u_\la)=E(u)$, making (CW) special and critical. It also gives rise to the {\bf ground state} solutions in the explicit form
\EQ{ \label{def W}
 \pt W_\la(x):=\la^{d/2-1}W(\la x), \pq W(x):=\left[1+\frac{|x|^2}{d(d-2)}\right]^{1-d/2}\in\dot H^1(\R^d),
 \pr \forall\la>0, \pq -\De W_\la+f'(W_\la)=0,}
which has the minimal energy among all the stationary solutions. The scale and translation invariance of (CW) generates a family of ground states as a smooth manifold in $\cH$ with dimension $1+d$: 
\EQ{ \label{def Static}
 \pt W_{\la,c}(x):=\la^{d/2-1}W(\la(x-c)) \implies  -\De W_{\la,c}+f'(W_{\la,c})=0,
 \pr \Static(W):=\{\vec W_{\la,c}\in\cH \mid \la>0,\ c\in\R^d\}.}
Then the Lorentz invariance generates a family of solitons on a smooth manifold in $\cH$ with dimension $1+2d$: 
\EQ{ \label{def Soliton}
 \pt u(t,x)=W_{\la,c,p}(t,x):=W_{\la,c+pt}(x+(\LR{p}-1)p|p|^{-2}p\cdot x) \implies (CW).
 \pr \Soliton(W):=\{\vec W_{\la,c,p}(0)\in\cH \mid \la>0,\ c\in\R^d,\ p\in\R^d\}.}

Other types of solutions are the {\bf scattering} (to $0$) solutions with the property
\EQ{
 \exists \fy\in\cH,\ \|\vec u(t)-U(t)\fy\|_\cH \to 0 \pq(t\to\I)}
where $U(t)$ denotes the free propagator, defined as the Fourier multiplier
\EQ{ \label{def U}
 U(t):=\mat{\cos(t|\na|) & |\na|^{-1}\sin(t|\na|) \\ -|\na|\sin(t|\na|) & \cos(t |\na|)}, \pq |\na|:=\sqrt{-\De},} 
the norm blow-up (called {\bf type-I blow-up} in \cite{DKM4})
\EQ{
 \limsup_{t\nearrow t^*}\|\vec u(t)\|_\cH=\I,}
and the more subtle {\bf type-II blow-up}, for which $\|\vec u(t)\|_\cH$ is bounded but $\vec u(t)$ fails to be strongly continuous in $\cH$ beyond some $t<\I$. 

In \cite{CNW-nonrad}, the authors gave a partial classification of dynamics of (CW) in the region 
\EQ{ \label{energy region}
 E(\vec u)<\sqrt{(E(W)+\e^2)^2+|P(\vec u)|^2}}
for a small $\e>0$, which is, by the Lorentz invariance, reduced to the region
\EQ{
 E(\vec u)<E(W)+\e^2.}
It was proved that if $u\in C([0,T_+);\cH)$ is a strong solution up to the maximal existence time $T_+\in(0,\I]$, which does not stay close to the ground state solitons near $t=T_+$, then $u$ either blows up away from the ground state solitons, or it scatters (to $0$) as $t\to\I$. We have the same for $t<0$, and moreover, the $2\times 2$ combinations of scattering and blow-up in $t>0$ and in $t<0$ respectively are realized by initial-data sets in $\cH$ which have non-empty interior. The key ingredient for proof is the existence of a small neighborhood of the ground states such that any solution exiting from it can never come back again, called the {\bf one-pass theorem}. 

A missing piece in the above result of \cite{CNW-nonrad} is the global dynamics around the ground states, compared with the corresponding results for the subcritical Klein-Gordon equation \cite{NLKG-nonrad} and for the Schr\"odinger equation in the radial symmetry \cite{NLS}, where we have $3\times 3$ complete classification of \eqref{energy region} including the {\bf scattering to the ground states} on a center-stable manifold of codimension $1$. 

On the other hand, there have been many papers \cite{KS,KST,DM,HR,DK,Beceanu,DHKS} for (CW) constructing various types of solutions around the ground states, including center-stable manifolds in some stronger topology than the energy space, on which the solutions scatter to the ground states, type-II blow-up at prescribed power law rate or with eternal oscillations between such rates, type-II blow-up at time infinity. The latter phenomena clearly distinguish (CW) from the dynamics of the subcritical equation. 

In this paper, we construct a smooth center-stable manifold of codimension $1$ in the energy space, which embraces all the solutions scattering to, or staying close to the ground state solitons. Indeed the last property is the defining characterization of the manifold. Plugging it into the result in \cite{CNW-nonrad}, we complete the $3\times 3$ classification for (CW) in the region \eqref{energy region}, which is now described. 

Denote $\cH$-distance to the ground states by 
\EQ{ \label{def distW}
 \dist_W(\fy):=\inf\{\|\vec u(t)-\psi\|_\cH\mid \text{$\psi$ or $-\psi$}\in\Static(W)\},}
and the time inversion for any initial data $\fy\in\cH$ and any initial data set $A\subset\cH$ by  
\EQ{ \label{def da}
 \fy^\da:=(\fy_1,-\fy_2), \pq A^\da:=\{\fy^\da\mid \fy\in A\}.}
\begin{thm} \label{main thm1}
There exist positive constants $\e<\de<1<C$, and an unbounded connected $C^1$ manifold $\M\subset\cH$ with codimension $1$ satisfying the following. $\Static(W)\\\subset\M$ is tangent to the center-stable subspace of the linearized equation at each point of $\Static(W)$. $\M$ is invariant by the flow, translation, rotation and the $\cH$-invariant scaling. Let $u$ be any solution with $E(\vec u(0))<E(W)+\e^2$, and let $T\in(0,\I]$ be its maximal existence time. Then we have only one of the following (1)--(3). 
\begin{enumerate}
\item $\vec u(0)\not\in\pm\M$, $T=\I$ and $\Lim_{t\to\I}\|\vec u(t)-\vec v(t)\|_\cH\to 0$ for some free solution $v$. 
\item $\vec u(0)\not\in\pm\M$, $T<\I$ and $\Liminf_{t\to T}\dist_W(\vec u(t))>\de>C\e$. 
\item $\vec u(0)\in\pm\M$ and $\Limsup_{t\to T}\dist_W(\vec u(t))\le C\sqrt{E(\vec u)-E(W)}$.
\end{enumerate}
Let $A_1,A_2,A_3$ be the corresponding sets of the initial data $\vec u(0)$. Then $A_1\cap A_1^\da$ is a non-empty open set. $A_1\cap A_2^\da$, $A_2\cap A_1^\da$ and $A_2\cap A_2^\da$ have non-empty interior. $A_1$ and $A_2$ have non-empty interior in $\pm\M^\da$ in the relative topology. $\M\cap\M^\da\subset\cH$ is a  connected $C^1$ manifold with codimension $2$. $\M\cap-\M=\emptyset=\M^\da\cap-\M$. 
\end{thm}
The above case (3) also contains blow-up solutions, but they are distinguished from (2) by the asymptotic distance from the ground states. By the characterization of type-II blow-up by Duyckaerts, Kenig and Merle \cite[Theorem 1]{DKM2}, in case (3) with $T<\I$, there are a smooth $(\la,c):[0,T)\to(0,\I)\times\R^d$ and $\fy\in\cH$ such that 
\EQ{
 \la(t)\to\I, \pq \|\vec u(t)-\vec W_{\la(t),c(t)}-\vec \fy\|_\cH\to 0}
as $t\nearrow T$. The gap between $C\e$ and $\de$ is actually huge in the proof. The above theorem except for the existence of $\M$ was essentially proved in \cite{CNW-nonrad}. 

Next we can exploit the Lorentz transform to include all the ground state solitons.
\begin{thm} \label{main thm2}
There exist a small constant $\e>0$, a connected $C^1$ manifold $\M_L\subset\cH$ with codimension $1$, and two open sets $O_1,O_2\subset\cH$ satisfying the following. $\M\cup\Soliton(W)\subset \M_L$. $\Soliton(W)\subset O_1\subset\overline{O_1}\subset O_2$. 
$\M_L$ is invariant by the flow, translation, $\cH$-invariant scaling and the Lorentz transform. 
Let $u$ be any solution with $E(\vec u(0))<\sqrt{|E(W)+\e^2|^2+|P(\vec u(0))|^2}$ and let $T\in(0,\I]$ be its maximal existence time. Then we have only one of the following (1)--(3). 
\begin{enumerate}
\item $\vec u(0)\not\in\pm\M_L$, $T=\I$ and $\Lim_{t\to\I}\|\vec u(t)-\vec v(t)\|_\cH\to 0$ for some free solution $v$. 
\item $\vec u(0)\not\in\pm\M_L$, $T<\I$ and $\vec u(t)\not\in\pm\overline{O_2}$ for all $t$ near $T$. 
\item $\vec u(0)\in\pm\M_L$ and $\vec u(t)\in \pm O_1$ for all $t$ near $T$. 
\end{enumerate}
The initial data sets and $\M_L$ enjoy the same properties as in the previous theorem. 
\end{thm}

The manifolds $\M$ and $\M_L$ are center-stable manifolds of $\Static(W)$ and $\Soliton(W)$ respectively, but they contain solutions blowing up in finite time. The invariance by the flow should be understood that the solutions starting on the manifold stay there as long as they exist, and similarly for the Lorentz transform. Again by \cite{DKM2}, in case (3) with $T<\I$, we have a smooth $(\la,c)$, $p\in\R^d$ and $\fy\in\cH$ such that, as $t\nearrow T$,  
\EQ{
 \la(t)\to\I, \pq \|\vec u(t)-\vec W_{\la(t),c(t),p}-\fy\|_\cH \to 0.}

Our argument to construct the center-stable manifold is somehow similar to the numerical bisection in \cite{S,Bizon}, where the center-stable manifold was searched for as the threshold between scattering and blowup. Indeed, our proof does not touch the delicate dynamics of those solutions on the manifold, but relies on the behavior of those off the manifold. In particular, we do not need any dispersive estimate on the linearized operator as in \cite{KS,NLKG-nonrad,NLS,Beceanu}, which makes our proof much simpler. In this respect, it is similar to \cite{Inv} in the subcritical case. On the other hand, the criticality or the concentration phenomenon forces us to work in the space-time rescaled according to the solution itself. For that part we employ the same argument as in the previous paper \cite{CNW-nonrad}.

The next question is if we can remove the energy restriction \eqref{energy region}. Concerning it, Duyckaerts, Kenig and Merle \cite{DKM4} recently established an outstanding result of {\it asymptotic soliton resolution} for $d=3$: Every solution with radial symmetry is either type-I blow-up, or decomposes into a sum of ground states with time-dependent scaling and a free solution
\EQ{ \label{DKM dcp}
 \exists N\ge 0,\ \exists \la_j(t),\ \exists\fy\in\cH, \pq \lim_{t\nearrow T}\|\vec u(t) - \sum_{j=1}^N \vec W_{\la_j(t),0} - U(t)\fy\|_\cH=0,}
where $T$ is the maximal existence time of $u$. 
Given this expansion, one might expect that the dividing manifold of dynamics could be extended as the collection of all such solutions with $N>0$. 
However, it is very hard to prove such a statement even if we know the above asymptotics, because of the instability of the ground state. Moreover, one can easily observe that the above naive guess is not correct when $T<\I$ and the energy is larger, as one can construct such blow-up solutions in the deep interior of blow-up solutions, by using finite speed of propagation (see Appendix \ref{sect:bup int}).

Instead of pursuing that approach, we extend our center-stable manifold globally by adding large radiation, thereby including at least all solutions \eqref{DKM dcp} with $N=1$ and $T=\I$, as well as some of them with $T<\I$. 
A simple procedure is proposed to reduce the analysis to the previous case $E<E(W)+\e^2$ by {\it detaching large radiation}, which relies on the asymptotic Huygens principle, valid for all $d\in\N$ and without radial symmetry. The extended manifold splits the energy space into the scattering and blow-up regions locally around itself, although the entire dynamical picture is still far beyond our analysis. 
\begin{thm} \label{main thm3}
There exist a connected $C^1$ manifold $\M_D\subset\cH$ with codimension $1$ and two open sets $O_3,O_4\subset\cH$ satisfying the following. $\M_L\subset\M_D$. $\Soliton(W)\subset O_3\subset \overline{O_3}\subset O_4$. $\M_D$ is invariant by the flow, translation, $\cH$-invariant scaling and the Lorentz transform. Let $u$ be any solution with $\vec u(0)\in O_3$. Then we have only one of the following (1)--(3). 
\begin{enumerate}
\item $\vec u(0)\not\in\pm\M_D$, $T=\I$ and $\Lim_{t\to\I}\|\vec u(t)-\vec v(t)\|_\cH=0$ for some free solution $v$.
\item $\vec u(0)\not\in\pm\M_D$, $T<\I$ and $\vec u(t)\not\in\pm\overline{O_4}$ for $t$ near $T$.
\item $\vec u(0)\in\pm\M_D$ and $\vec u(t)\in\pm O_3$ for all $t$ near $T$. 
\end{enumerate}
If $u\in C([0,\I);\cH)$ is a solution satisfying
\EQ{ \label{scat2soliton}
 \Lim_{t\to\I}\|\vec u(t)-\vec W_{\la(t),c(t),p(t)}-\vec v(t)\|_\cH=0,}
for some $(\la,c,p):[0,\I)\to(0,\I)\times\R^{1+2d}$ and a free solution $v$, then $\vec u(0)\in\M_D$. 
\end{thm}
More detailed statements are given in the main body of paper.

This paper is organized as follows. In the rest of this section, we introduce some notation and coordinates, together with a few basic facts and estimates, mostly overlapping with the previous paper \cite{CNW-nonrad}. 

In {\bf Part I} starting with Section 2, we deal with the solutions with energy slightly above the ground state. The center-stable manifold is constructed as a threshold between scattering and blowup, which completes the 9-set classification of dynamics, in a form similar to the subcritical case \cite{NLKG-nonrad}. The main new ingredient is the {\it ignition Lemma} \ref{lem:inst}, which roughly says that for any solution staying close to the ground states, any arbitrarily small perturbation in the unstable direction eventually leads to the ejection from a small neighborhood as in the ejection lemma of \cite{CNW-nonrad}. These are extended by the Lorentz transform in the end of Section 3.

In {\bf Part II} starting with Section 4, we extend the results in Part I to large energy by adding out-going radiation. The extended manifold contains all the solutions scattering to the ground state solitons, while it is still a dynamical threshold between the scattering and the blowup. The main ingredient is the {\it detaching Lemma} \ref{lem:detach}, which allows one to detach out-going radiation energy from a solution to produce another solution with smaller energy but the same behavior. We also extend the one-pass theorem of \cite{CNW-nonrad} by allowing out-going large radiation.

\subsection{Strichartz norms and strong solutions}
For any $I\subset\R$, we use the Strichartz norms for the wave equation with the following exponents 
\EQ{ \label{def St}
 \pt \St_s:=L^{q_s}_t(I;\dot B^{1/2}_{q_s,2}(\R^d)),
 \pq \St_m:=L^{q_m}_{t,x}(I\times \R^d), 
 \pq \St_s^*:=L^{q_s'}_t(I;\dot B^{1/2}_{q_s',2}(\R^d)),
 \pr \St_p:=L^{(d+2)/(d-2)}_t(I;L^{2(d+2)/(d-2)}(\R^d))\pq (d\le 6),
 \pr \St(I):=\St_s(I)\cap\St_m\cap\St_p(I), 
 \pq q_s:=\frac{2(d+1)}{d-1},\ q_m:=\frac{2(d+1)}{d-2},}
where $q':=q/(q-1)$ denotes the H\"older conjugate. 
Slightly abusing the notation, we often apply these norms to the first component of vector functions such as
\EQ{
 \|(v_1,v_2)\|_{\St(I)}:=\|v_1\|_{\St(I)}.}
The small data theory using the Strichartz estimate implies that there is a small $\e_S>0$ such that for any $T>0$ and $\fy\in\cH$ satisfying 
\EQ{ \label{def eS}
 \vec v(t):=U(t)\fy \implies \|v\|_{\St_m(0,T)} \le \e_S,}
there is a unique solution $u$ of (CW) on $[0,T)$ satisfying
\EQ{
 \vec u(0)=\fy,\pq \pt\|\vec u-\vec v\|_{L^\I\cH(0,T)}+\|u-v\|_{\St(0,T)} 
  \pn\lec \|f'(u)\|_{\St_s^*(0,T)} 
  \pr\lec \|u\|_{\St_m(0,T)}^{2^*-2}\|u\|_{\St_s(0,T)} \ll \|u\|_{\St(0,T)} \sim \|v\|_{\St(0,T)},}
which is scattering to $0$ if $T=\I$. The uniqueness holds in
\EQ{ \label{loc St}
 \{\vec u \in C([0,T);\cH)\mid \forall S\in(0,T),\ u\in \St(0,S)\},} 
and a solution $u$ in this space can be extended beyond $T<\I$ if and only if $\|u\|_{\St_m(0,T)}<\I$. 
For any interval $I\subset\R$, we denote by 
\EQ{ \label{def Sol}
 \Sol(I)} 
the set of all solutions $u$ of (CW) on $I$ such that $\vec u\in C(J;\cH)\cap\St(J)$ for any compact $J\subset I$, and that {\it $u$ can not be extended beyond any open boundary of $I$.}  
For example, $\vec u\in\Sol([0,T))$ with $T<\I$ means that $u$ is a solution in \eqref{loc St} with $\|u\|_{\St(0,T)}=\I$, so that it can be extended to $t<0$ but not to $t>T$, whereas $\vec u\in\Sol([0,T])$ means that $u$ is a solution in $C([0,T];\cH)\cap\St(0,T)$, which can be extended both to $t<0$ and to $t>T$.  
 Hence, if $I$ is an open interval, then $I$ is the maximal life for any $u\in\Sol(I)$. 

\subsection{Symmetry, solitons and linearization}
The groups of space translation and scaling are denoted by 
\EQ{ \label{def TS}
 \pt \T^c \fy(x):= \fy(x-c) \pq (c\in\R^d), 
 \pr \cS^\s(\fy_1,\fy_2)(x):=(e^{\s(d/2-1)}\fy_1(e^\s x),e^{\s d/2}\fy_2(e^\s x))\ (\s\in\R),}
and for any $a\in\R$, $(S_a^\s \fy)(x):=e^{\s(d/2+a)}\fy(e^\s x)$. Their generators are $\T'=-\na$, $\cS'=S_{-1}'\otimes S_0'$, and $S_a'=r\p_r+d/2+a$. 
For any $\vec u\in\Sol(I)$ and $(\s,c)\in\R^{1+d}$, 
\EQ{
 \T^c\cS^\s \vec u(e^\s t) \in \Sol(e^{-\s}I).}

The linearization around the ground state $W$ is written by the operator
\EQ{ \label{def L+}
 L_+ := -\De - f''(W), \pq f''(W)=(2^*-1)(1+|x|^2/(d(d-2)))^{-2},}
as well as the nonlinear term 
\EQ{ \label{def N}
 N(v):=f'(W+v)-f'(W)-f''(W)v=O(v^2).}
The matrix version of the linearization is given by $J\cL$ with 
\EQ{ \label{def cL}
  \cL:=\mat{L_+ & 0 \\ 0 & 1}, \pq J=\mat{0 & 1 \\ -1 & 0}.}
The generator of each invariant transform of (CW) gives rise to generalized null vectors, namely with $A=S_{-1}'$ or $\T'=-\na$, 
\EQ{
  J\cL \mat{A W \\ 0}=0, \pq J\cL \mat{0 \\ A W}=\mat{AW \\ 0}. }
It is well known that for the ground state $W$, there is no other generalized null vector. Note however that $A W$ is not an eigenfunction but a threshold resonance, i.e.~$A W\not\in L^2$ for $d\le 4$. 
Besides $L_+^{-1}(0)=\{\na W\}$ and the absolutely continuous spectrum $[0,\I)$, $L_+$ has only one negative eigenvalue and the ground state,  
\EQ{ \label{def ro}
 L_+\ro = -k^2 \ro, \pq 0<\ro\in H^2(\R^d),\pq k>0, \pq \|\ro\|_2=1,}
for which the orthogonal subspace and projection are denoted by 
\EQ{ \label{def Horth}
 \cH_\perp:=\{\fy\in\cH \mid \LR{\fy|\ro}=0\in\R^2\}, \pq P_\perp:=1-\ro\langle\ro|.}
Henceforth, the $L^2$ inner product is denoted by 
\EQ{ \label{def L2prod}
 \LR{\fy|\psi}:=\re\int_{\R^d} \fy(x)\overline{\psi(x)}dx \in\R,}
and for vector functions 
\EQ{
 \pt \LR{(\fy_1,\fy_2)|(\psi_1,\psi_2)}:=\LR{\fy_1|\psi_1}+\LR{\fy_2|\psi_2}\in\R,
 \pr \LR{(\fy_1,\fy_2)|\psi}:=(\LR{\fy_1|\psi},\LR{\fy_2|\psi})=:\LR{\psi|(\fy_1,\fy_2)}\in\R^2,}
which may be applied to column vectors as well as higher dimensional vectors. Throughout the paper, a pair with a comma $(\cdot,\cdot)$ denotes a vector, but never an inner product. 

\subsection{Coordinates around the ground states}
We recall from \cite{CNW-nonrad} our dynamical coordinates for the solution $\vec u$ around the ground states $\Static(W)$: 
\EQ{ \label{decop u}
 \vec u(t) = \T^{c(t)}\cS^{\s(t)}(\vec W+ v(t)), \pq v(t)=\la(t)\ro +\ga(t),\pq \CAS{\la(t)=\LR{v(t)|\ro},\\ \ga(t)=P_\perp v(t),}}
where $v(t)=(v_1(t),v_2(t))\in\cH$ does not generally satisfy $v_2(t)=\dot v_1(t)$ because of the modulation $(\s(t),c(t))$. The unstable and stable modes are denoted by $\la_\pm$:
\EQ{ \label{def la+-}
 \la\ro= \la_+g_+ + \la_-g_-,\pq \la_\pm:=\sqrt{\frac{k}{2}}\la_1\pm\sqrt{\frac{1}{2k}}\la_2, \pq g_\pm:=\frac{1}{\sqrt{2k}}(1,\pm k)\ro,}
for which we introduce linear functionals $\La_\pm:\cH\to\R$ by 
\EQ{ \label{def La+-}
 \La_\pm\fy=\frac{1}{\sqrt{2k}}\LR{\fy|(k,\pm1)\ro}=\LR{J\fy|\mp g_\mp},}
so that we have 
\EQ{
 v=\La_+(v)g_++\La_-(v)g_-+ P_\perp v.} 
If $\vec u(t)$ is close to $\Static(W)$, then we can uniquely choose $(\s(t),c(t))$ such that the orthogonality condition holds\footnote{The sign of $\mu$ is switched from \cite{CNW-nonrad} for better notational symmetry.}
\EQ{ \label{orth}
  \R^{1+d} \ni (\al(t),\mu(t)):=\LR{v_1(t)|(S_0',\T')\ro}=0}
by the implicit function theorem. Note that it is not preserved by the linearized equation, since neither $S_0'\ro$ nor $\T'\ro$ is an eigenfunction of $L_+$. 
 The linearized energy norm $E$ is defined on the entire $\cH$ by 
\EQ{ \label{def E}
 \| v\|_E^2
 \pt:=k^2\la_1^2+\la_2^2+\LR{\cL\ga|\ga}+|\al|^2+|\mu|^2 \sim \| v\|_\cH^2
 \pr=|k\LR{v_1|\ro}|^2+|\LR{v_2|\ro}|^2+\LR{\cL P_\perp v| v}+|\LR{\ga_1|S_0'\ro}|^2+|\LR{\ga_1|\T'\ro}|^2
 \pr=k|\la_+|^2+k|\la_-|^2+\LR{\cL P_\perp v| v}+|\LR{v_1|S_0'\ro}|^2+|\LR{v_1|\T'\ro}|^2.}
See \cite[Lemma 2.1]{CNW-nonrad} for a proof of the equivalence to $\cH$. 

Note that all the above are static operations in $\cH$ defined around $\Static(W)$. More precisely, we define the bi-continuous affine maps $\Phi_{\s,c}:\R^2\times\cH_\perp\to\cH$ and $\Psi_{\s,c}:\cH\to\cH$ for each $(\s,c)\in\R^{1+d}$ by
\EQ{ \label{def PhiPsi}
 \pt \Phi_{\s,c}(\la,\ga)=\Psi_{\s,c}(\la\ro+\ga), \pq \Psi_{\s,c}(v)=\T^c\cS^\s(\vec W+v).}
For any $\de>0$, define open neighborhoods of $0$ and $\Static(W)$ in $\cH$, by 
\EQ{ \label{def BN}
 \B_\de:=\{v\in\cH \mid \|v\|_\cH<\de\},
 \pq \cN_\de:=\bigcup_{(\s,c)\in\R^{1+d}}\Psi_{\s,c}(\B_\de)\subset\cH.}
Then $\{(\Psi_{\s,c}(\B_\de),\Psi_{\s,c}^{-1})\}_{(\s,c)\in\R^{1+d}}$ is an atlas for the open set $\cN_\de\subset\cH$. Here is a precise statement on the orthogonality \eqref{orth}
\begin{lem} \label{lem:orth}
There exist $\de_\Phi\in(0,1)$ and a smooth map $(\ti\s,\ti c):\cN_{\de_\Phi}\to\R^{1+d}$ such that for any $(\s,c)\in\R^{1+d}$ and $\fy\in\Psi_{\s,c}(\B_{\de_\Phi})$, we have $\LR{\Psi_{\s,c}^{-1}(\fy)_1|(S_0',\T')\ro}=0$ if and only if $(\s,c)=(\ti\s(\fy),\ti c(\fy))$, and moreover
\EQ{ \label{est tisc}
 |\ti\s(\fy)-\s|+e^{\s}|\ti c(\fy)-c| \lec \|P_\perp \Psi_{\s,c}^{-1}(\fy)_1\|_{\dot H^1}.}
\end{lem}
\begin{proof}
For any $\psi=\la\ro+\ga\in\B_\de$ and $(\s,c)\in\R^{1+d}$, define $(\al,\mu):\R^{1+d}\to\R^{1+d}$ by
\EQ{
 (\al,\mu)(s,y)\pt:=\LR{\Psi_{s,y}^{-1}(\Psi_{\s,c}(\psi))_1|(S_0',\T')\ro}
 \pr=\LR{S_{-1}^{\s-s}\T^{e^{\s}(c-y)}(W+\psi_1)-W|(S_0',\T')\ro},}
where we used the identity $\T^c  S_a^\s =  S_a^\s \T^{e^\s c}$. 
Hence we have 
\EQ{ \label{der s almu}
 \pt|(\al,\mu)(\s,c)|=|\LR{\ga_1|(S_0',\T')\ro}| \lec \|\ga_1\|_{\dot H^1} \lec \de,
 \pr\p_s(\al,\mu)=\LR{W+\psi_1|\T^{e^{\s}(y-c)}S_1^{s-\s}S_1'(S_0',\T')\ro}
 \pr\pq\pq=(-b_W,0)+O(\de+|s-\s|+e^{\s}|y-c|),}
where $b_W:=\LR{S_{-1}'W|S_0'\ro}=k^{-2}(2^*-1)(2^*-2)\LR{W^{2^*-3}(S_{-1}'W)^2|\ro}>0$ (see \cite[(2.26)]{CNW-nonrad} for the identity), and
\EQ{ \label{der y almu}
 e^{-\s}\p_y(\al,\mu)\pt=\LR{W+\psi_1|\T' \T^{e^{\s}(y-c)} S_1^{s-\s}(S_0',\T')\ro}
 \pr=-(0,a_WI)+O(\de+|s-\s|+e^{\s}|y-c|),}
where $a_W:=\LR{-\De W|\ro}/d=\LR{f'(W)|\ro}/d>0$ and $I$ denotes the identity matrix acting on $\R^d$. Then the implicit function theorem implies that there is a unique $(s,y)\in\R^{1+d}$ such that 
\EQ{
 (\al,\mu)(s,y)=0, \pq |s-\s|+e^{\s}|y-c|\lec \|\ga_1\|_{\dot H^1} \lec \de,} 
provided that $\de>0$ is small enough. Since $(\al,\mu)$ is obviously smooth in $\psi$, the implicit function is also smooth in $\Psi_{\s,c}(\B_\de)$. For the uniqueness on $\cN_\de$, suppose that $\Psi_{\s,c}(\psi)\in\Psi_{s,y}(\B_\de)$ for some $(s,y)\in\R^{1+d}$, then 
\EQ{
 \de\gec\|\Psi_{\s,c}(0)-\Psi_{s,y}(0)\|_\cH=\|(\T^c\cS^\s-\T^y\cS^s)\vec W\|_\cH \sim |s-\s|+e^{\s}|y-c|.}
Hence the uniqueness on $\cN_\de$ follows from the implicit function theorem. 
\end{proof}

For brevity, we define $\sT_\fy:\cH\to\cH$ for $\fy\in\cN_{\de_\Phi}$, and $\ti\la:\cN_{\de_\Phi}\to\R^2$ by
\EQ{ \label{def sT}
 \sT_\fy:=\T^{\ti c(\fy)}\cS^{\ti\s(\fy)}, 
 \pq \ti\la(\fy):=\LR{\sT_\fy^{-1}(\fy)-\vec W|\ro},}
and similarly $\ti\la_\pm:\cN_{\de_\Phi}\to\R$. 

\begin{rem}
$\Phi_{\s,c}$ and $\Psi_{\s,c}$ are not smooth in $(\s,c)$ for the $\ga$ component, since the derivative in $(\s,c)$ induces $\cS'\ga$ and $\T'\ga$, which are not generally in $\cH$. Indeed $\Phi_{\s,c}$ is continuous for $(\s,c)$ at each fixed point on $\cH$, but not uniformly on any ball in $\cH$. In \cite{Inv} this was remedied by introducing a topology (``mobile distance") in which translations are also Lipschitz continuous. Instead of that, we will fix $(\s,c)$ with respect to perturbation of the initial data, even though we modulate it in time. 
\end{rem}

Next we change the time variable from $t$ to $\ta$ by
\EQ{ \label{def ta}
 \ta(0)=0, \pq \frac{d\ta}{dt}=e^{\s(t)}.}
Then we get the equation of $v$ as an evolution in $\ta$:
\EQ{ \label{eq vta}
 \pt \p_\ta v= J\cL v + \U N(v_1) - Z(\vec W+v), 
 \prq \U N(\fy):=(0,N(\fy)), \pq Z=(Z_1,Z_2):=\s_\ta\cS'+e^\s c_\ta\cdot \T',}
and differentiating the orthogonality condition \eqref{orth} yields 
\EQN{ 
 \pt 0=\p_\ta\al=\LR{\p_\ta v_1|S_0'\ro}=\LR{\ga_2|S_0'\ro} + e^\s c_\ta\LR{v_1|\T'S_0'\ro} - \s_\ta[b_W-\LR{v_1|S_1'S_0'\ro}],
 \pr 0=\p_\ta\mu=\LR{\p_\ta v_1|\T'\ro}=\LR{\ga_2|\T'\ro} - e^\s c_\ta[a_WI-\LR{v_1|\na^2\ro}] + \s_\ta \LR{v_1|S_1'\T'\ro}. }
Hence, as long as $v_1$ is small, 
\EQ{ \label{eq scta}
 \mat{\s_\ta \\ e^\s c_\ta}
 \pt= \mat{b_W-\LR{v_1|S_1'S_0'\ro} & -\LR{v_1|\T'S_0'\ro} \\ -\LR{v_1|S_1'\T'\ro} & a_WI-\LR{v_1|\na^2\ro}}^{-1}\mat{\LR{\ga_2|S_0'\ro} \\ \LR{\ga_2|\T'\ro}} 
 \pr= (1+O(\|v_1\|_{\dot H^1}))\mat{b_W^{-1}\LR{\ga_2|S_0'\ro} \\ a_W^{-1}\LR{\ga_2|\T'\ro}}. }
This is linear in $v$ (or $\ga$), because the orthogonality \eqref{orth} is not preserved by the linearized equation, a notable difference from the standard modulation analysis in the subcritical case. In the original time $t$, it yields 
\EQ{ \label{eq sct}
 (e^{-\s}\s_t,c_t)=(1+O(\|v_1\|_{\dot H^1}))\LR{\ga_2|(S_0'\ro/b_W,\T'\ro/a_W)}.} 
For the eigenmode we have
\EQ{ \label{eq la12}
 \p_\ta\la \pt= \mat{0 & 1 \\ k^2 & 0}\la  - \LR{ Z v|\ro}
 + \LR{\U N(v_1)|\ro}
 \pr=\mat{\la_2 + \s_\ta (\al+\la_1) + e^\s c_\ta \mu  \\ k^2\la_1 + \LR{N(v_1)-Z_2\ga_2|\ro}}
 \pn= \mat{\la_2 + \s_\ta \la_1 \\ k^2\la_1 + \LR{N(v_1)-Z_2\ga_2|\ro}},
}
where we used $\al=0=\mu$ only in the last step, since we will consider the case $(\al,\mu)\not=0$ as well. 
In the unstable/stable modes, the equation reads 
\EQ{ \label{eq la+-}
 \p_\ta \la_\pm = \pm k \la_\pm + \sqrt{\frac{k}{2}}[\s_\ta(\al+\la_1)+e^\s c_\ta \mu] \pm \sqrt{\frac{1}{2k}} \LR{N(v_1)-Z_2\ga_2|\ro}.}
We also recall the distance function $d_W:\cH\to[0,\I)$ defined in \cite{CNW-nonrad}, which satisfies $d_W(\fy)\sim \dist_W(\fy)$, and, for some constant $\de_E>0$, if $d_W(\fy)\le\de_E$ then $\pm\fy\in\cN_{\de_\Phi}$ and 
\EQ{ \label{def dW}
 d_W(\fy)^2=E(\fy)-E(W)+k^2\ti\la_1(\pm\fy)^2}
for either sign $\pm$. 

\vspace{36pt}

{\large \bf Part I: Slightly above the ground state energy}\\

In the first part of paper, we study the global dynamics in the region $E(u)<E(W)+\e^2$, and its Lorentz extension, completing the picture in \cite{CNW-nonrad} with a center-stable manifold and the dynamics around it. 

\section{Center-stable manifold around the ground states}
First we construct a center-stable manifold around the ground states $\Static(W)$. This will be later extended in three ways: 
\begin{enumerate}
\item By the backward flow, to the region $E<E(W)+\e^2$, 
\item By the Lorentz transform, to the region $E<\sqrt{E(W)^2+\e^4+|P(\vec u)|^2}$
\item By adding large radiation, which may have arbitrarily large energy.  
\end{enumerate}
In order to define the manifold as a graph of $(\la_-,\ga)\mapsto \la_+$, we define 
\EQ{ \label{def B+'}
 \pt \B_\de^+:=\{\la_+\in\R\mid \|\la_+ g_+\|_\cH<\de\},
 \pr \B_\de':=\{\la_-g_-+\ga \mid \la_-\in\R,\ \ga\in\cH_\perp,\ \|\la_- g_-+\ga\|_\cH<\de\},}
then $\B_{\de/C} \subset \B_\de^+g_+\oplus\B_\de'\subset \B_{C\de}$ for some constant $C>1$. The corresponding neighborhood of $\Static(W)$ is denoted by
\EQ{ \label{def N2}
 \cN_{\de_1,\de_2}:=\{\Psi_{\s,c}(\la_+g_++\fy)\mid (\s,c)\in\R^{1+d},\ \la_+\in\B^+_{\de_1},\ \fy\in\B'_{\de_2}\}.}
\begin{thm} \label{thm:cmd0}
There exist constants $\de_m,\de_X>0$ satisfying $\de_m\ll\de_X\ll\de_\Phi$ and $\de_m\ll \e_S$, and a unique $C^1$ function $m_+:\B_{\de_m}'\to\B_{\de_m}^+$ with the following property. Let $\la_+\in \B_{\de_m}^+$, $\fy\in\B_{\de_m}'$,  $(\s,c)\in\R^{1+d}$, $T>0$ and $\vec u\in\Sol([0,T))$ with $\vec u(0)=\Psi_{\s,c}(\la_+g_+ + \fy)$. Then we have the trichotomy:
\begin{enumerate}
\item \label{bup} If $\la_+>m_+(\fy)$, then $u$ blows up away from the ground states. More precisely
\EQ{
 T<\I, \pq \liminf\limits_{t\nearrow T}d_W(\vec u(t))>\de_X}
or $\liminf_{t\nearrow T}\dist_{L^{2^*}}(u(t),\Static(W)_1)\gg\de_m$. 
\item \label{trap} If $\la_+=m_+(\fy)$, then $u$ is trapped by the ground states. More precisely, $d_W(\vec u(t))$ is decreasing until it reaches
\EQ{ \label{sink}
 d_W(\vec u(t))^2 \le 2(E(u)-E(W)) \le 2d_W(\vec u(0))^2 \lec\|\fy\|_\cH^2,}
and stays there for the rest of $t<T$. 
\item \label{scat} If $\la_+<m_+(\fy)$, then $u$ scatters to $0$. More precisely, 
\EQ{
 T=\I,\pq \exists \fy_\I\in\cH,\ \lim\limits_{t\to\I}\|\vec u(t)-U(t)\fy_\I\|=0.}
\end{enumerate}
Moreover, in the cases \eqref{bup} and \eqref{scat}, there exists $T_X\in(0,T)$ such that 
\EQ{ \label{exiting}
 \pt 0<t<T_X \implies d_W(\vec u(t))<\de_X,
 \pq T_X<t<T \implies d_W(\vec u(t))>\de_X,
 \pr \ti\la_+(\vec u(T_X))\sim\ti\la_1(\vec u(T_X))\sim-K(u(T_X))\sim\pm\de_X,}
with the sign $+$ for \eqref{bup} and $-$ for \eqref{scat}. In addition, $m_+(0)=m_+'(0)=0$ and 
\EQ{ \label{tan of La}
 |m_+(\fy^1)-m_+(\fy^2)| \lec (\|\fy^1\|_\cH+\|\fy^2\|_\cH)^{1/6}\|\fy^1-\fy^2\|_\cH.}
\end{thm}
Obviously, the three asymptotics in \eqref{bup}--\eqref{scat} are distinctive. 
From the preceding results around the ground states, we know that the case \eqref{trap} contains type-II blowup and global solutions scattering to the ground states. 
Type-I blowup is contained in the case \eqref{bup}, but it may also contain type-II blowup. 
\eqref{exiting} comes from the one-pass theorem proved in \cite{CNW-nonrad}. 

Thus we obtain a manifold of codimension $1$ in $\cH$: 
\EQ{ \label{def M0}
 \M_0:=\{\Psi_{\s,c}(m_+(\fy)g_++\fy)\mid \fy\in \B_{\de_m}',\ (\s,c)\in\R^{1+d}\},}
which contains $\Static(W)$ and is invariant by the forward flow within $\cN_{\de_m,\de_m}$. It is also invariant by $\T$ and $\cS$ by definition. 
 
Then \eqref{tan of La} implies that it is tangent to the center-stable subspace of the linearized evolution at each point on $\Static(W)=\{\Psi_{\s,c}(0)\}_{\s,c}\subset \M_0$, and that $\M_0$ is transverse to its time inversion $\M_0^\da$, since $\fy\mapsto\fy^\da$ exchanges $\la_+$ and $\la_-$. More explicitly 
\EQ{
 \M_0\cap\M_0^\da=\Cu_{(\s,c)}\Phi_{\s,c}\{(\la,\ga)\mid (\la_\pm,\ga)\in\B_{\de_m}',\ \la_\pm=m_+(\la_\mp g_++(\ga_1,\pm\ga_2))\}}
is a local center manifold of codimension $2$, on which every solution $u$ satisfies 
\EQ{
 \de_m\gec d_W(\vec u(t)) \gg |\ti\la(\vec u(t))|} 
all over its life, though it is not necessarily global. Obviously, $\M_0^\da$ is relatively closed in $\M_0$, splitting it into two non-empty, relatively open sets where $\la_->m_+(\la_+ g_+ + \ga^\da)$ or $\la_-<m_+(\la_+ g_+ + \ga^\da)$. The solutions starting from the first set blow up away from the ground states in $t<0$, while those starting from the second set scatters to $0$ as $t\to-\I$. 

A $C^1$ functional $M_+:\cN_{\de_m,\de_m}\to\R$ is defined such that $\M_0=M_+^{-1}(0)$, by putting
\EQ{ \label{def M+}
 M_+(\fy)=\la_+-m_+(\la_-g_-+\ga),\pq \fy=\Phi_{\ti\s(\fy),\ti c(\fy)}(\la,\ga).}
It is clearly non-degenerate in the direction $\sT_\fy g_+$ by \eqref{tan of La}.  Moreover, $M_+>0$, $M_+=0$ and $M_+<0$ respectively give the trichotomy \eqref{bup}--\eqref{scat}. 

The proof of the above theorem goes as follows. First we observe that if $\|\fy\|_\cH\ll|\la_+|\ll 1$ then we can apply the ejection lemma and the one-pass theorem from \cite{CNW-nonrad}, and obtain \eqref{bup} for $\la_+\gg\|\fy\|_\cH$ and \eqref{scat} for $-\la_+\gg\|\fy\|_\cH$. Moreover, the ejection lemma implies that every solution ejected at $t=T_X$ from a small neighborhood of $\Static(W)$ is categorized either in \eqref{bup} with $\ti\la_+(\vec u(T_X))>0$, or in \eqref{scat} with $\ti\la_+(\vec u(T_X))<0$. So each set of such initial data is open in $\cH$. Hence there is at least one $\la_+$ in between, for which the solution is never ejected, i.e.~the case \eqref{trap}. The uniqueness of such $\la_+$ follows also from the instability of $W$, or the exponential growth of the unstable component $\la_+$ for the difference of two solutions. The next section is devoted to its estimate, which is essentially the only ingredient in addition to \cite{CNW-nonrad}.

As in \cite{Inv}, we abbreviate the differences by the following notation:
\EQ{ \label{def diff}
 \diff X^\pa := X^1 - X^0, \pq \diff F(X^\pa):=F(X^1)-F(X^0),}
for any symbol $X$ and any function $F$. 

\subsection{Igniting the unstable mode} \label{sect:ignit}
In this subsection, we prove the following: For any solution trapped by the ground states, an arbitrarily small perturbation leads to the ejection from the small neighborhood unless the perturbation is almost zero in the unstable direction. More precisely,

\begin{lem}[Ignition lemma] \label{lem:inst}
There exist constants $0<\ig_I<1<C_I<\I$ with the following property. Let $T>0<\ig\le\ig_I$, $\vec u^0\in\Sol([0,T))$, $(s,y)\in\R^{1+d}$ and $\fy\in\cH$ satisfy $\vec u^0(0)\in\Psi_{s,y}(\B_\ig)$, 
\EQ{ \label{init pert}
 \pt \|d_W(\vec u^0)\|_{L^\I_t(0,T)}< \ig^3, \pq C_I \ig \|\fy\|_\cH < |\La_+\fy|, \pq \|\fy\|_\cH < \ig^6.}
Then there exist $t_I\in(0,T)$, $\la_+\in\R$, and $\vec u^1\in\Sol([0,t_I])$ such that 
\EQ{
 \pt \vec u^1(0)=\vec u^0(0)+ \T^y\cS^s \fy,
 \pr \|\vec u^0-\vec u^1\|_{L^\I(0,t_I;\cH)}=\|\vec u^0(t_I)-\vec u^1(t_I)\|_\cH=\ig^3 \sim \la_+\sign(\La_+\fy),  
 \pr \|\vec u^0(t_I)-\vec u^1(t_I)-\sT_{\vec u^0(t_I)}\la_+g_+\|_\cH \lec \ig^4.}
In particular, we have 
\EQ{ \label{ignited}
 d_W(\vec u^0(t_I))+d_W(\vec u^1(t_I))\sim\ig^3.}
\end{lem}
This lemma is proved by exponential growth in the unstable direction of the difference $u^0-u^1$ in the rescaled coordinate for $u^0$. It may take very long depending on the initial size of the perturbation, but in the rescaled time $\ta$, where the solution $u^0$ is (forward) global in both the scattering and the blow-up cases. The difference is estimated mainly by the energy argument, rather than dispersive estimates.  
The nonlinearity is too strong to be controlled solely by Sobolev, for which we employ Strichartz norms which are uniform on unit intervals of $\ta$. 

Hence the main idea is similar to \cite{Inv}, but we do not use the mobile distance, but instead the same modulation parameters $(\s(t),c(t))$ for both $u^0$ and $u^1$, in order to avoid destroying the energy structure for the difference. This is indeed much simpler, whereas the former idea seems hard to apply in the critical setting because of the change of time variable. 

The main difference from the ejection lemma in \cite{CNW-nonrad} is that there is no bound on the time for the unstable mode to grow to some amount, and we estimate the difference of two solutions rather than the difference from the ground state. In particular, the equation for the difference naturally contains linear terms coming from the nonlinearity, which prevents us from a crude Duhamel argument as in \cite{CNW-nonrad}. 

Before starting the proof, we see that the Strichartz norms can be uniformly bounded on unit time intervals in the rescaled variables: 
\begin{lem}[locally uniform perturbation in $\St_\ta$] \label{lem:LUS}
There is a constant $\y_l\in(0,1]$ with the following property. Let $T>0<\de\le\y_l$, $\vec u\in\Sol([0,T))$, 
\EQ{
 \pt \vec u(t)=\Psi_{\s(t),c(t)}v(\ta(t)), \pq \frac{d\ta}{dt}(t)=e^{\s(t)}, \pq \ta(0)=0,
 \prq \|v(0)\|_\cH+|e^{-\s(t)}\s'(t)|+|c'(t)| \le \de, }
for $0<t<T$. Then we have $\ta(T)>\y_l$ and 
\EQ{
 \|v\|_{\St\cap L^\I\cH(0,\y_l)} \lec \de.}
Moreover, if $\Sol([0,T^1))\ni \vec u^1(t)=\Psi_{\s(t),c(t)}v^1(\ta(t))$ satisfies $\|\vec u^1(0)-\vec u(0)\|_\cH<\y_l$, then we have $\ta(T^1)>\y_l$ and  
\EQ{
 \|v^1-v\|_{\St\cap L^\I\cH(0,\y_l)} \lec \|\vec u^1(0)-\vec u(0)\|_\cH.}
\end{lem}
\begin{proof}
We obtain from the inequality on $(\s',c')$ that 
\EQ{
 |e^{\s(t)}-e^{\s(0)}|<e^{\s(t)}/4<e^{\s(0)}/2, \pq |\s(t)-\s(0)|+e^{\s(t)}|c(t)-c(0)|<\de}
for $0<t<e^{-\s(0)}/4$. Let $\vec W^0:=\T^{c(0)}\cS^{\s(0)}\vec W$ and $\vec w(t):=\vec u(t)-\vec W^0$. Then we have 
\EQ{
 \|W^0\|_{\St(0,e^{-\s(0)}\y)}=\|W\|_{\St(0,\y)} \lec \y^{1/q_m}} 
for $0<\y<1$, and $w$ solves on $(0,T)$
\EQ{
 (\p_t^2-\De)w=f'(W^0+w)-f'(W^0).}
Hence by Strichartz we have for small $\y>0$ 
\EQ{
 \|\vec w-\vec w_F\|_{\St\cap L^\I\cH(0,e^{-\s(0)}\y)} \pt\ll \|w\|_{\St(0,e^{-\s(0)}\y)}
 \prq\sim \|w_F\|_{\St(0,e^{-\s(0)}\y)} \lec \|\vec w(0)\|_\cH = \|v(0)\|_\cH,}
where $\vec w_F:=U(t)\vec w(0)=U(t)v(0)$. In particular there is some $\y_l<1/8$ such that the above estimates hold for $\y\le 2\y_l$, and so 
$\ta(T)>\ta(2e^{-\s(0)}\y_l)>\y_l$.
Under the scaling property 
\EQ{
 1/p + d/q - s = d/2 -1,}
we have, for any $T'\in(0,T)$, 
\EQ{ \label{St t2ta}
 \pt\|v_1\|_{L^p_\ta(0,\ta(T');\dot B^s_{q,2})}^p
 \pn= \int_0^{T'} \|v_1(\ta(t))\|_{\dot B^s_{q,2}}^p e^{\s(t)}dt
 \pr= \int_0^{T'} \|\T^{c(t)}S_{-1}^{\s(t)}v_1(\ta(t))\|_{\dot B^s_{q,2}}^p dt
 = \|u_1-\T^{c(t)}S_{-1}^{\s(t)}W\|_{L^p_t(0,T';\dot B^s_{q,2})}^p,}
and similarly for $v_2$ and in $L^\I\cH$. Hence putting $T'=\ta^{-1}(\y_l)<2e^{-\s(0)}\y_l$, 
\EQ{
 \|v\|_{\St\cap L^\I\cH(0,\y_l)}
 \le \|\vec w\|_{\St\cap L^\I\cH(0,T')}+\|(\T^{c}\cS^{\s}-\T^{c(0)}\cS^{\s(0)})\vec W\|_{\St\cap L^\I\cH(0,T')},}
and the last norm is bounded by 
\EQ{
 \pt\|(\T^{c-c(0)}-I)\cS^{\s}\vec W\|_{\St\cap L^\I\cH(0,T')} + \|(\cS^{\s-\s(0)}-I)\cS^{\s(0)}\vec W\|_{\St\cap L^\I\cH(0,T')}
 \pr\lec \|e^\s|c-c(0)|+|\s-\s(0)|\|_{L^\I(0,T')} \lec \de,}
which concludes the estimate on $v$. 

For the difference, the same change of variable as in \eqref{St t2ta} yields for $T'\in(0,T)$
\EQ{
 \|v^1-v\|_{\St\cap L^\I\cH(0,\ta(T'))}
 =\|\vec u^1-\vec u\|_{\St\cap L^\I\cH(0,T')}.}
Since $\|u\|_{\St(0,\ta^{-1}(\y_l))}+\|\vec u(0)-\vec u^1(0)\|_\cH\ll 1$, we have $T^1>\ta^{-1}(\y_l)$ and 
\EQ{
 \|\vec u^1-\vec u\|_{\St\cap L^\I\cH(0,\ta^{-1}(\y_l))} \lec \|\vec u^1(0)-\vec u(0)\|_\cH,}
which leads to the estimate on $v^1-v$. 
\end{proof}

\begin{proof}[Proof of Lemma \ref{lem:inst}]
Let $(\s,c):[0,T)\to\R^{1+d}$ be the modulation for $u^0$ defined by 
\EQ{
 (\s(t),c(t))=(\ti \s(\vec u^0(t)),\ti c(\vec u^0(t))),}
and let $\ta:[0,T)\to(0,\I)$ be the rescaled time variable for $u^0$ defined by
\EQ{
 \frac{d\ta}{dt}=e^{\s(t)}, \pq \ta(0)=0.}
Let $\vec u^1\in\Sol([0,T^1))$ be the solution with the initial data 
\EQ{
 \vec u^1(0)=\vec u^0(0)+\T^y\cS^s \fy.}
We use the same coordinates for $u^0$ and $u^1$ by putting for $j=0,1$ and at each $t$
\EQ{
 \vec u^j=\Psi_{\s,c}(v^j)=\Phi_{\s,c}(\la^j,\ga^j).} 
For the initial perturbation, we have, using \eqref{est tisc}, 
\EQ{
 \pt|\diff\la^\pa(0)-\LR{\fy|\ro}|
 =|\LR{(\cS^{s-\s(0)}\T^{e^s(y-c(0))}-I)\fy|\ro}|
 \pr\lec (|s-\s(0)|+e^s|y-c(0)|)\|\fy\|_\cH
 \pr\lec \|(\Psi_{s,y})^{-1}(\vec u^0(0))\|_\cH \|\fy\|_\cH
 \lec \ig \|\diff v^\pa(0)\|_\cH,}
which, together with \eqref{init pert}, implies
\EQ{
 C_I\ig\|\diff v^\pa(0)\|_\cH < 2|\diff\la^\pa_+(0)|,}
if $C_I$ is large and $\ig_I$ is small enough. 

Since the modulation $(\s,c)$ was chosen for $u^0$, we have 
\EQ{
 \|v^0\|_{L^\I(0,T;\cH)} \sim \sup_{0<t<T}d_W(\vec u^0)=:\de < \ig^3.}
For the Strichartz norms, Lemma \ref{lem:LUS} implies 
\EQ{ \label{St v}
 \|v^0\|_{\St(\ta_0<\ta<\ta_0+\y_l)} \lec \de,
 \pq \|\diff v^\pa\|_{\St(\ta_0<\ta<\ta_0+\y_l)} \lec \|\diff v^\pa(\ta_0)\|_\cH,}
as long as $\|\diff v^\pa\|_\cH$ remains small, while \eqref{eq scta} implies 
\EQ{
 |\ta_0-\ta_1|\lec \de^{-1} \implies |\s(\ta_0)-\s(\ta_1)|+|e^{\s(\ta_0)}[c(\ta_0)-c(\ta_1)]| \lec \de|\ta_0-\ta_1|.}
In particular, we have $\ta(t)\nearrow\I$ as $t\nearrow T$, since otherwise $|\s|$ is bounded as $t\nearrow T$, and so, if $T<\I$ then $\|u^0\|_{\St(0,T)}<\I$ contradicting the blowup at $T$, and if $T=\I$ then the boundedness of $\dot\ta=e^{\s}$ implies that $\ta\to\I$. 

In the following, we regard all the dynamical variables as functions of $\ta$ rather than $t$, unless explicitly specified. 
We have the equations for the difference
\EQ{ \label{eqdiff}
 \pt \p_\ta\diff\la^\pa_1=\diff\la^\pa_2 + \s_\ta(\al^1+\diff\la^\pa_1) + e^\s c_\ta\cdot \mu^1,
 \pr \p_\ta\diff\la^\pa_2=k^2\diff\la^\pa_1 + \diff{\LR{N(v^\pa_1)-Z_2\ga^\pa_2|\ro}}, 
 \pr \p_\ta(\al^1,\mu^1) = \LR{\diff\ga^\pa_2|(S_0',\T')\ro} - \LR{Z_1\diff v^\pa_1|(S_0',\T')\ro},
 \pr \p_\ta\diff\ga^\pa = J\cL\diff\ga^\pa + P_\perp[\diff{\U N(v^\pa_1)}- Z\diff{ v^\pa}],}
Also remember that $(\al^0,\mu^0)=\LR{v^0_1|(S_0',\T')\ro}=0$ and so 
\EQ{
 (\al^1,\mu^1)=\diff(\al^\pa,\mu^\pa)=\LR{\diff\ga^\pa_1|(S_0',\T')\ro}.}
Hence the third equation follows from the fourth one in \eqref{eqdiff}. 
By the assumption, 
\EQ{
 \de < \ig^3, \pq \|\diff v^\pa(0)\|_\cH < \ig^6.}
Suppose that for some $\ta_0>0$ we have 
\EQ{ \label{cone cond}
 \pt \|\diff v^\pa\|_{L^\I(0,\ta_0;\cH)} < \ig^3,\pq \ig^2|\diff\la^\pa_-(\ta_0)|+\ig\nu(\ta_0)<|\diff\la^\pa_+(\ta_0)|,}
where we put 
\EQ{ \label{def nu}
 \nu(\ta):=\sqrt{\ig^2|(\al^1,\mu^1)|^2+\LR{\cL\ga|\ga}}.} 
Choosing $C_I>1$ large enough, we have \eqref{cone cond} at $\ta_0=0$. We will prove that the second condition of \eqref{cone cond} is preserved until the first one is broken.

The linearized energy in \eqref{def E} implies that at each time 
\EQ{
 \|\diff\ga^\pa\|_\cH \gec \nu \gec \ig \|\diff\ga^\pa_1\|_{\dot H^1} + \|\diff\ga^\pa_2\|_{L^2} + \ig|(\al^1,\mu^1)|.}
Lemma \ref{lem:LUS} implies that $u^1$ exists at least for $\ta<\ta_0+\y_l$ and 
\EQ{
 \|v^0\|_{\St\cap L^\I\cH(\ta_0,\ta_0+\y_l)} \lec \de,  \pq \|\diff v^\pa\|_{\St\cap L^\I\cH(\ta_0,\ta_0+\y_l)} \lec \|\diff v^\pa(\ta_0)\|_\cH < \ig^3.}
Using \eqref{eq scta} as well, we derive from the difference equations 
\EQ{
 \pt |(\p_\ta\mp k)\diff\la^\pa_\pm| \lec \de\|\diff v^\pa\|_\cH + \de|(\al^1,\mu^1)|+\|\diff v^\pa\|_\cH^2 \lec \ig^3|\diff\la^\pa|+\ig^2\nu,
 \pr |\p_\ta(\al^1,\mu^1)| \lec \|\diff\ga^\pa_2\|_{L^2}+\de\|\diff v^\pa\|_\cH.}

To control $\diff\ga^\pa$, we use the linearized energy identity 
\EQ{
 \pt\p_\ta\LR{\cL\diff\ga^\pa|\diff\ga^\pa}/2
 =\LR{\diff{\U N(v^\pa_1)}- Z\diff v^\pa|\cL\diff\ga^\pa}
 \pr=\LR{\diff N(v^\pa_1)|\diff\ga^\pa_2}
 +e^\s c_\ta[\LR{L_+\na\ro|\diff\ga^\pa_1}\diff\la^\pa_1+\LR{\na\pot|(\diff\ga^\pa_1)^2/2}+\LR{\na\ro|\diff\ga^\pa_2}\diff\la^\pa_2]
 \pr\pq\pq -\s_\ta[\LR{L_+ S_{-1}'\ro|\diff\ga^\pa_1}\diff\la^\pa_1
 +\LR{S_{2-d/2}'\pot|(\diff\ga^\pa_1)^2/2} 
 + \LR{S_0'\ro|\diff\ga^\pa_2}\diff\la^\pa_2],}
where $\pot:=f''(W)$. 
The terms on the right except for the first one are simply bounded by $\de\|\diff\ga^\pa\|_\cH\|\diff{ v}^\pa\|_\cH$.  The nonlinear term can be bounded only via $\ta$-integral. For any interval $I=(\ta_0,\ta_1)\subset[0,\I)$ with $|I|<\y_l$, we have, 
\EQ{
 \int_I|\LR{\diff N(v^\pa_1)|\diff\ga^\pa_2}|d\ta
 \lec \|\diff\ga^\pa_2\|_{L^\I_\ta(I;L^2_x)}\|\diff N(v^\pa_1)\|_{L^1_\ta(I;L^2_x)}.}
Since 
\EQ{
 \diff N(v^\pa_1)=\int_0^1[f''(W+v^\te_1)-f''(W)]\diff v^\pa_1 d\te, \pq v^\te:=(1-\te)v^0+\te v^1,}
we have\footnote{Here for simplicity we use the exponents available only for $d\le 6$, but it is clear that we only need H\"older continuity of $f''$, i.e.~$2^*>2$, and so it can be easily modified for all dimensions $d\ge 3$.} 
\EQ{ \label{diff N est}
 \|\diff N(v^\pa_1)\|_{L^1L^2(I)}
 \pt\lec \sup_{0<\te<1} \|[f''(W+v^\te_1)-f''(W)]\diff v^\pa_1\|_{L^1_\ta L^2_x(I)}
 \pr\lec (\|v^0_1\|_{\St_p(I)}+\|\diff v^\pa_1\|_{\St_p(I)})\|\diff v^\pa_1\|_{\St_p(I)} 
 \pr\lec (\de+\|\diff v^\pa(\ta_0)\|_\cH) \|\diff v^\pa(\ta_0)\|_{\cH},}
where we used Lemma \ref{lem:LUS} in the last step. We thus obtain 
\EQ{
 |[\LR{\cL\diff\ga^\pa|\diff\ga^\pa}]_{\ta_0}^{\ta_1}| \lec \ig^3\|\diff\ga^\pa\|_{L^\I_\ta(\cH)}\|\diff{ v}^\pa(\ta_0)\|_{\cH}.}
Combining this with the estimate on $(\al^1,\mu^1)$ yields  
\EQ{ \label{nu incest}
 [\nu^2]_{\ta_0}^{\ta_1} \pt\lec \ig\nu[\|\diff\ga^\pa_2\|_{L^2_x} + \de\|\diff v^\pa(\ta_0)\|_\cH] 
  + \ig^3\|\diff\ga^\pa\|_{L^\I_\ta \cH}\|\diff{ v}^\pa(\ta_0)\|_{\cH}
 \pr\lec \ig\|\nu\|_{L^\I_\ta}^2 + \ig^2\|\nu\|_{L^\I_\ta}\|\diff{ v}^\pa(\ta_0)\|_{\cH}
 \pn\lec \ig\|\nu\|_{L^\I_\ta}^2 + \ig^3 \|\diff\la^\pa\|_{L^\I_\ta}^2.}
Thus we obtain for $\ta_0<\ta<\ta_0+\y_l$, 
\EQ{ \label{diff est1}
 \pt \nu \le (1+C\ig)\nu(\ta_0)+C\ig^{3/2}\|\diff\la^\pa\|_{L^\I_\ta},
 \pr |(\p_\ta\mp k)\diff\la^\pa_\pm| \lec \ig^3|\diff\la^\pa|+\ig^2\nu.}
Suppose that $\|\diff\la^\pa\|_{L^\I_\ta}\le M|\diff\la^\pa(\ta_0)|$. Then using that $\ig\nu(\ta_0)<|\diff\la^\pa_+(\ta_0)|$, 
\EQ{
 |(\p_\ta\mp k)\diff\la^\pa_\pm| \lec [\ig^3M+\ig(1+\ig^{3/2}M)]|\diff\la^\pa(\ta_0)| \lec \ig |\diff\la^\pa(\ta_0)|,}
provided that $\ig^{3/2}M\ll 1$. Hence by continuity in $\ta$, we deduce that 
\EQ{
 |\diff\la^\pa|\lec|\diff\la^\pa(\ta_0)|,} 
for $\ta_0<\ta<\ta_0+\y_l$, and plugging this into \eqref{diff est1}, 
\EQ{ \label{diff est2}
 \pt \nu \le (1+C\ig) \nu(\ta_0) + C\ig^{3/2} |\diff\la^\pa(\ta_0)|,
 \pr |\diff\la^\pa_\pm-e^{\pm k(\ta-\ta_0)}\diff\la^\pa_\pm(\ta_0)| \le C\ig|\diff\la^\pa_+(\ta_0)|.}
In particular, if $0<\ig\ll 1$ then
\EQ{
  [|\diff\la^\pa_+|]_{\ta_0}^{\ta_0+\y_l} \pt> (e^{(k-C\ig)\y_l}-1)|\diff\la^\pa_+(\ta_0)| 
 \pr\gg \ig|\diff\la^\pa_+(\ta_0)| \gec [\ig^2|\diff\la^\pa_-(\ta)|+\ig\nu]_{\ta_0}^{\ta_0+\y_l}.}
Hence the last condition of \eqref{cone cond} is transferred to $\ta=\ta_0+\y_l$. Therefore by iteration, \eqref{cone cond} and the above estimates hold with $\ta_0=n\y_l$ for integers $0\le n<N$, where either $N=\I$ or $2\le N<\I$ and 
$\|v\|_{L^\I_\ta(N\y_l,(N+1)\y_l;\cH)} \ge \ig^3$. 

We can improve the above estimates as follows. First from the second estimate for $\la_+$ in \eqref{diff est2}, we have for all $0\le n<N$, 
\EQ{
 \pt e^{(k-C\ig)\y_l}<\diff\la^\pa_+((n+1)\y_l)/\diff\la^\pa_+(n\y_l)<e^{(k+C\ig)\y_l},}
which precludes the case $N=\I$. Plugging this exponential growth in the same estimate for $\la_-$, we obtain for $0\le n\le N$
\EQ{ 
 |\diff\la^\pa_-(n\y_l)| \le e^{-(k-C\ig)n\y_l}|\diff\la^\pa_-(0)|   + C\ig|\diff\la^\pa_+(n\y_l)|.}
Using those two in the first estimate of \eqref{diff est2}, we obtain of $0\le n\le N$
\EQ{
 \nu(n\y_l) \le e^{C\ig n\y_l}\nu(0)+C\ig^{3/2}[|\diff\la^\pa_-(0)|+|\diff\la^\pa_+(n\y_l)|].}
Iterating once again, we obtain continuous versions for $0<\ta<N\y_l$
\EQ{
 \pt e^{(k-C\ig)\ta} \lec \diff\la^\pa_+(\ta)/\diff\la^\pa_+(0) \lec e^{(k+C\ig)\ta}, 
 \pr |\diff\la^\pa_-(\ta)| \lec e^{-(k-C\ig)\ta}|\diff\la^\pa_-(0)|   + \ig|\diff\la^\pa_+(\ta)|, 
 \pr \|\nu\|_{L^\I(0,\ta)} \lec e^{C\ig \ta}\nu(0) + \ig^{3/2}[|\diff\la^\pa_-(0)|+|\diff\la^\pa_+(\ta)|].}
In particular, we have
\EQ{
 \|\diff v^\pa\|_{L^\I(0,\ta;\cH)} \pt\lec \|\diff\la^\pa\|_{L^\I(0,\ta)}+\ig^{-1}\|\nu\|_{L^\I(0,\ta)} 
 \pr\lec |\diff\la^\pa_+(\ta)| + \ig^{-1}e^{C\ig\ta}\|\diff v^\pa(0)\|_\cH.}
Since $\|\diff v^\pa(0)\|_\cH<\ig^6$, the last term is bounded by $\ig^{4}$ for $e^{C\ig\ta}< \ig^{-1}$, while 
\EQ{
 |\diff\la^\pa_+(\ta)|\gec \left[e^{C\ig\ta}\right]^{(k-C\ig)/C\ig}\ig^{2}\|\diff v^\pa(0)\|_\cH \gg \ig^{-10}e^{C\ig\ta}\|\diff v^\pa(0)\|_\cH,}
for $e^{C\ig\ta}\ge\ig^{-1}$, choosing $\ig_I\ll k$. Hence for some $\ta_I\in(N\y_l,(N+1)\y_l)$ we have
\EQ{
 \pt \ig^3 =\|\diff v^\pa(\ta_I)\|_\cH=\|\diff v^\pa\|_{L^\I(0,\ta_I;\cH)}\sim|\diff\la^\pa_+(\ta_I)|\sim|\diff\la^\pa_+(\ta_I/2)|,
 \pr \ig^4 \gec \|\diff\la^\pa_-\|_{L^\I(0,\ta_I)}+\|\nu/\ig\|_{L^\I(0,\ta_I)} 
 \prq\gec \|\diff\la^\pa_-\|_{L^\I(0,\ta_I)} + \|\diff\ga^\pa\|_{L^\I(0,\ta_I;\cH)} + \|(\al^1,\mu^1)\|_{L^\I(0,\ta_I)}.}
$\ta_I<\I$ means that $t_I:=\ta^{-1}(\ta_I)\in(0,T)$ and $\vec u^1\in\Sol([0,t_I])$. 
For the last statement of the lemma, suppose $d_W(\vec u^0(t_I))\ll\ig^3$. Then $\|\ga^1(t_I)\|_\cH\lec\ig^4+d_W(\vec u^0(t_I))\ll\ig^3$ and via \eqref{est tisc} we have $|\diff{\ti\s(\vec u^\pa(t_I))}|+e^{\s(t_I)}|\diff{\ti c(\vec u^\pa(t_I))}|\lec \|\ga^1(t_I)\|_\cH\ll\ig^3$, and so 
\EQ{
 d_W(\vec u^1(t_I))\gec \|\vec v^1(t_I)\|_\cH-C\|\sT_{\vec u^0(t_I)}\vec W-\sT_{\vec u^1(t_I)}\vec W\|_\cH \gec \ig^3.}
Therefore $d_W(\vec u^0(t_I))+d_W(\vec u^1(t_I))\sim\ig^3$ in any case (the upper bound is obvious). 
\end{proof}

\subsection{Construction of the manifold}
Now we are ready to prove Theorem \ref{thm:cmd0}. Let $0<\de'\ll\de_+\ll\de_\Phi$, $\la_+\in\B^+_{\de_+}$, $\fy\in\B'_{\de'}$ and $\vec u\in\Sol([0,T))$ with $\vec u(0)=\la_+g_++\fy$. Choosing $\de_+$ small enough ensures that 
\EQ{
 E(u)<E(W)+\e_*^2, \pq \de_+ \ll \de_*,}
where $\de_* \gg \e_*>0$ are the small constants in the one-pass theorem \cite[Theorem 5.1]{CNW-nonrad}. For each fixed $\fy$, we divide the set $\B^+_{\de_+}$ for $\la_+$ according to the behavior of $\vec u$. Let $A_\pm$ be the totality of $\la_+\in\B^+_{\de_+}$ for which there exists $t_0\in[0,T)$ such that 
\EQ{ \label{eject cond}
 \pt \de_*^2>d_W(\vec u(t_0))^2>2(E(u)-E(W)), 
 \pr \p_td_W(\vec u(t_0))> 0,\pq \pm \ti\la_+(\vec u(t_0))>0.}
Then the ejection lemma \cite[Lemma 3.2]{CNW-nonrad} followed by the one-pass theorem \cite[Theorem 5.1]{CNW-nonrad} implies the following. If $\la_+\in A_\pm$ then the solution $\vec u$ is exponentially ejected out of the small neighborhood $d_W<d_W(\vec u(t_0))$ and never comes back again. Moreover, if $\la_+\in A_+$ then $u$ blows up in $t>t_0$, and if $\la_+\in A_-$ then $u$ scatters to $0$. 
By the local wellposedness of (CW) in $\cH$, both $A_\pm$ are open. To see that both are non-empty, consider the case $\de_+\sim|\la_+|\gg\|\fy\|_\cH$. Let 
\EQ{
 \vec u(t)=\sT_{\vec u(t)}(\vec W+\ti\la(t)\ro+\ga(t)), \pq \ga(t)\perp\ro.}
Then \eqref{est tisc} implies that 
\EQ{
 |\ti\la_+(0)-\la_+|+|\ti\la_-(0)| \lec \|\fy\|_\cH \ll |\la_+|,}
and, by definition of $d_W$, we have at $t=0$ 
\EQ{
 \pt d_W(\vec u)^2=E(u)-E(W)+k^2|\ti\la_1|^2, 
 \pr e^{-\ti\s}\p_t d_W(\vec u)^2=\frac{k^2}{2}(|\ti\la_+|^2-|\ti\la_-|^2)+O(\|\ga\|_\cH|\ti\la_1|^2),
 \pr E(u)-E(W)=-k\ti\la_+\ti\la_-+\frac{1}{2}\LR{\cL\ga|\ga}-o(|\ti\la|^2+\|\ga\|_\cH^2).}
Hence we deduce 
\EQ{
 |E(u)-E(W)|\ll |\la_+|^2 \sim d_W(\vec u(0))^2 \sim e^{-\ti\s}\p_td_W(\vec u(0))^2>0,}
and thus $\la_+\in A_{\sign\la_+}$ for $|\la_+|\sim\de_+$. In particular, both $A_\pm$ are non-empty, which implies that the remainder 
\EQ{
 A_0:=\B^+_{\de_+}\setminus(A_+\cup A_-)} 
is also non-empty. Every solution $u$ for $\la_+\in A_0$ must violate \eqref{eject cond} for each $t\in[0,T)$, to avoid the ejection. Hence at each $t\in(0,T)$, one of the following holds 
\begin{enumerate}
\item $\p_t d_W(\vec u(t))\le 0$ \label{approaching}
\item $d_W(\vec u(t))^2\le 2(E(u)-E(W)) \le 2d_W(\vec u(0))^2$ \label{trapped}
\item $d_W(\vec u(t))\ge \de_*$, \label{ejected}
\end{enumerate}
where the last condition in \eqref{eject cond} is not considered, since it is implied by the others, due to the ejection lemma. 
Since $d_W(\vec u(0))\lec \de_+\ll\de_*$, either \eqref{approaching} or \eqref{trapped} holds for small $t>0$. Since the ejection lemma can be applied with $\p_t d_W(\vec u)=0$ as well, we deduce that $d_W(\vec u(t))$ is strictly decreasing in $t>0$ until \eqref{trapped} is satisfied, where $\vec u$ spends its remaining life (hence never reaching \eqref{ejected}). 

Choosing $\de_+\ll\ig_I^6$ small enough, we can ensure that 
\EQ{
  d_W(\vec u(0))^2 + |\la_+| \ll \ig_I^6}
for all $(\la_+,\fy)\in\B^+_{\de_+}\times\B'_{\de'}$. If there are more than one $\la_+\in A_0$ for the same $\fy$, say $\la_+^0\not=\la_+^1$, then we can apply Lemma \ref{lem:inst} to the corresponding solutions with the initial data $\vec u^j(0)=\vec W+\la_+^jg_++\fy$ and with $\ig\in(0,\ig_I]$ satisfying 
\EQ{
 d_W(\vec u^j(0))^2+|\diff\la^\pa_+|\ll\ig^6.} 
Then its conclusion together with \eqref{trapped} leads to a contradiction 
\EQ{
 \ig^3 \sim d_W(\vec u^0(t_I))+ d_W(\vec u^1(t_I)) 
 \lec d_W(\vec u^0(0))+d_W(\vec u^1(0)) \ll \ig^3.}
Thus we can define the functional $m_+$ by putting $A_0=\{m_+(\fy)\}$, and then $A_+=(m_+(\fy),\de_+)$, $A_-=(-\de_+,m_+(\fy))$. The same reasoning as above implies that we can never apply Lemma \ref{lem:inst} for different $\fy^0,\fy^1\in\B'_{\de'}$ with $\vec u^j(0)=\vec W+m_+(\fy^j)g_++\fy^j$ and $d_W^2(\vec u^j(0))+\|\diff\fy^\pa\|_\cH\ll\ig^6$. Therefore 
\EQ{
 \max_{j=0,1}\left[ |m_+(\fy^j)|+\|\fy^j\|_\cH \right] \ll \ig^6 <\ig_I^6 \implies 
 |\diff m_+(\fy^\pa)| \lec \ig\|\diff\fy^\pa\|_\cH,}
which implies the Lipschitz continuity \eqref{tan of La}. Since $m_+(0)=0$ is obvious, it also implies that $|m_+(\fy)|=o(\|\fy\|_\cH)$. In particular, we may restrict both the domain and the range of $m_+$ to have the same radius $\de_m$ as in the statement of the theorem, though there is no merit for that besides reducing the number of parameters. 

The trichotomic dynamics readily follows from the ejection lemma and the one-pass theorem in \cite{CNW-nonrad}. The estimate on the $L^{2^*}$ distance in \eqref{bup} is derived from the bound $K(u(t))\le-\ka(\de_*)$ in the variational region \cite[Lemma 4.1]{CNW-nonrad} as follows. Choose $\de_m\ll\ka(\de_*)$. Let $u$ be a solution in the case \eqref{bup} and let $t$ be after the ejection, namely $d_W(\vec u(t))>\de_*$. Let $u(t)=\psi+\fy$ with $\psi=\T^c S_{-1}^\s W$ for some $(\s,c)\in\R^{1+d}$. Then 
\EQ{
 \ka(\de_*)\le -K(u(t))\pt=-K(\psi)-\|\na\fy\|_{L^2}^2+2\LR{\De\psi|\fy}+\|\psi+\fy\|_{L^{2^*}}-\|\psi\|_{L^{2^*}}
 \pr=-\|\na\fy\|_{L^2}^2-2\LR{\psi^{2^*-1}|\fy}+\|\psi+\fy\|_{L^{2^*}}-\|\psi\|_{L^{2^*}},}
which implies $\ka(\de_*)\lec\|\fy\|_{L^{2^*}}$ and so $\dist_{L^{2^*}}(u(t),\Static(W)_1)\gec\ka(\de_*)\gg\de_m$. 
Thus it only remains to prove that $m_+$ is $C^1$. 

\subsection{Smoothness of the manifold}
Next we prove that the center-stable manifold obtained above is at least $C^1$ in $\cH$. Let $\vec u^0\in\Sol([0,T))$ be a solution on the manifold with the modulation and the rescaled variables
\EQ{
 (\s,c)=(\ti\s,\ti c)(\vec u^0(t)),\pq \frac{d\ta}{dt}=e^{\s(t)}, \pq \ta(0)=0.} 
We consider solutions $\vec u^1$ close to $\vec u^0$, in the form  
\EQ{
 \pt \vec u^1(t) = \vec u^0(t)+\sT_{\vec u^0(t)} h \ck v(t,h), \pq \ck v(t,h)=\ck\la(t,h)\ro+\ck\ga(t,h),\pq \ck\ga\perp\ro,
 \pr \la_-^1(0)g_-+\ga^1(0):=\fy^0+h\fy',\pq (h\to 0),}
for each $\fy^0\in \B'_{\de_m}$ and $\fy'\in\B'_\I$. Let 
$\vec u^j = \Phi_{\s,c}(\la^j,\ga^j)$, then $(\ck\la,\ck\ga)=\diff(\la^\pa,\ga^\pa)/h$. 
To put both $\vec u^j$ on the manifold, we need 
\EQ{
 \la^0_+(0)=m_+(\fy^0), \pq \la^1_+(0)=m_+(\fy^0+h\fy').}
Let $(\ck\al,\ck\mu):=\LR{\ck\ga_1|(S_0',\T')\ro}$. 
From the equation \eqref{eqdiff} of $\diff{\vec u}^\pa$, we obtain 
\EQ{
 \pt \p_\ta\ck\la_1=\ck\la_2 + \s_\ta(\ck\al+\ck\la_1) + e^\s c_\ta \ck\mu,
 \pr \p_\ta\ck\la_2=k^2\ck\la_1 + \LR{\diff N(v^\pa_1)/h-Z_2\ck\ga_2|\ro},
 \pr \p_\ta(\ck\al,\ck\mu)=\LR{\ck\ga_2-Z_1\ck v_1|(S_0',\T')\ro},
 \pr \p_\ta\ck\ga=J\cL\ck\ga + P_\perp[\diff N(v^\pa_1)/h- Z\ck v].}
By Lemma \ref{lem:LUS}, $\ck v$ is bounded in $L^\I\cH\cap\St$ as $h\to 0$ locally uniformly on $0<\ta<\I$. Moreover, the small Lipschitz property of $m_+$ implies that 
\EQ{
 |\ck\la_+(0,h)|\ll \|\fy'\|_\cH.} 
Hence there is a sequence $h\to 0$ along which $\ck\la_+(0,h)$ converges to some $\la_+^\I\in\R$, and the limit of $\ck v=\ck\la\ro+\ck\ga\to v'=\la'\ro+\ga'$ satisfies the linearized equation 
\EQ{ \label{Leq}
 \pt \p_\ta\la'_1=\la'_2 + \s_\ta\la'_1 + \LR{\ga'_1|Z_1\ro},
 \pr \p_\ta\la'_2=k^2\la'_1 + \LR{N'(v_1^0)v'_1-Z_2\ga'_2|\ro},
 \pr \p_\ta\ga'=J\cL\ga'+P_\perp[N'(v_1^0)v'_1- Z  v'],}
where $N'(v_1^0):=f''(W+v_1^0)-f''(W)$, with the initial data 
$\la'(0)\ro+\ga'(0)=\la_+^\I g_++\fy'$. 
Regarding $(\la',\ga')$ as the unknown variables, this system is almost the same as \eqref{eqdiff} for $\diff{ v^\pa}$, except for $N'(v_1^0)v'_1$, which is the leading term of the nonlinearity in the latter system. 
In particular, we can show, by the same argument\footnote{Note that the proof of Lemma \ref{lem:inst} did not use any particular structure of $\diff N(v^\pa)$.} as for Lemma \ref{lem:inst}, that unless $|\la'_+(0)|$ is much smaller than $|\la'_-(0)|+\|\ga'(0)\|_\cH$, we have that $\la_+'(\ta)$ grows exponentially and so becomes dominant over the other components. More precisely, 
\begin{lem} \label{lem:instL}
There exists a constant $1<C_D<\I$ with the following property. 
Let $\vec u^0\in\Sol([0,T))$ satisfy\footnote{The constant $\ig_I$ is chosen here to be the same as in Lemma \ref{lem:inst}, just for convenience. It does not mean that the admissible range of $\ig_I$ is exactly the same for these two lemmas.} $\|d_W(\vec u^0)\|_{L^\I(0,T)}<\ig^3 \le \ig_I^3$. Let 
\EQ{
 (\s,c)=(\ti \s,\ti c)(\vec u),\pq v^0=\Psi_{\s,c}^{-1}(\vec u^0),\pq \frac{d\ta}{dt}=e^{\s(t)},\pq  \ta(0)=0.} 
Then equation \eqref{Leq} has a unique global solution $(\la'(\ta),\ga'(\ta)):[0,\I)\to\R^2\times\cH_\perp$ for any initial data $(\la'(0),\ga'(0))\in\R^{2}\times\cH_\perp$.  Moreover, if 
\EQ{ \label{cone cond L}
 C_D \ig \|(\la'_-(\ta_0),\ga'(\ta_0))\|_{\R\times\cH} \le |\la'_+(\ta_0)| }
at some $\ta_0\ge 0$, then there exists $\ta_I\in(\ta_0,\I)$, such that for all $\ta>\ta_I$ we have 
\EQ{ \label{grow la'}
 C_D|\la'_+(\ta)| > |\la'_+(\ta_I)|e^{k(\ta-\ta_I)/2} + \|(\la'_-(\ta),\ga'(\ta))\|_{\R\times\cH}/\ig.}
On the other hand, if \eqref{cone cond L} fails for all $\ta_0\ge 0$, then for all $\ta>0$
\EQ{ \label{slow grow bd}
 |\la'_+(\ta)| \lec \ig\|(\la_-'(\ta),\ga'(\ta))\|_{\R\times\cH} \lec e^{C\ig\ta}\|(\la_-'(0),\ga'(0))\|_{\R\times\cH}.}
\end{lem}
\begin{proof}
We only sketch the proof for \eqref{slow grow bd}, since the rest is essentially the same as Lemma \ref{lem:inst}. In the same way as for \eqref{diff est1},  we obtain, for any $0<\ta_0<\ta<\ta_0+\y_l$, 
\EQ{
 \pt \nu(\ta) \le (1+C\ig)\nu(\ta_0)+C\ig^{3/2}|\la'(\ta_0)|, 
 \pr |(\p_\ta+k)\la'_-(\ta)| \lec \ig^3|\la'(\ta_0)|+\ig^2\nu(\ta_0),}
where $\nu(\ta):=\sqrt{\ig^2|(\al',\mu')|^2+\LR{\cL\ga'|\ga'}}$ and $(\al',\mu'):=\LR{\ga'_1|(\cS',\T')\ro}$. Using that $|\la_+'|\lec \ig|\la_-'|+\nu$, we deduce from the above estimate 
\EQ{
 \pt \nu(\ta)+\ig^{1/2}|\la'_-(\ta)| \lec e^{C\ig\ta}[\nu(0) + \ig^{1/2}|\la'_-(0)|],}
and so, using \eqref{def E}, we obtain \eqref{slow grow bd}.
\end{proof}

The above lemma implies that for each $(\la'_-(0),\ga'(0))\in\R\times\cH_\perp$, there is a unique $\ck m_+(\la'_-(0),\ga'(0))\in\R$ such that if $\la_+'(0)=\ck m_+$ then \eqref{cone cond L} is not satisfied at any $\ta_0\ge 0$, and if $\pm(\la_+'(0)-\ck m_+)>0$ then $\pm\la_+'(\ta)$ grows exponentially to $\I$. 
To see that $\ck\la_+(0,h)\to\ck m_+$, we apply Lemma \ref{lem:inst} to $\vec u^0$ and $\vec u^1$. Since $\vec u^1(0)\to\vec u^0(0)$ as $h\to 0$, the local wellposedness implies that $\vec u^1\to\vec u^0$ locally uniformly on $[0,T)$, and so does $(\ti\s(\vec u^1),\ti c(\vec u^1))$. Hence for any $S<\I$, for $|h|$ small enough we could apply Lemma \ref{lem:inst} to $\vec u^0$ and $\vec u^1$ starting at any $\ta\in[0,S]$, if we had 
\EQ{
 2\ig_I\|\diff{\vec v}^\pa(\ta)\|_\cH<|\diff\la^\pa_+(\ta)|,}
but its conclusion would contradict that both $\vec u^j$ are on $\M_0$ (specifically between \eqref{sink} and \eqref{ignited}). Hence for all $\ta>0$, as $h\to 0$ along the sequence, 
\EQ{
 \ig_I\gec \frac{|\ck\la_+(\ta,h)|}{|\ck\la_-(\ta,h)|+\|\ck\ga(\ta,h)\|_\cH} \to \frac{|\la'_+(\ta)|}{|\la'_-(\ta)|+\|\ga'(\ta)\|_\cH},}
which implies $\la'_+(0)=\ck m_+$, since otherwise the right hand side will grow in $\ta$ at least to $O(1/\ig_I)$. 

Therefore $\ck\la_+(0,h)\to \ck m_+(\la_-'(0),\ga'(0))$ as $h\to 0$, without restricting $h$ to a sequence. In other words, $m_+$ is G\^{a}teaux differentiable at $\fy^0$
\EQ{
 \ck m_+(\la_-'(0),\ga'(0))\pt=\lim_{h\to 0}\frac{m_+(\fy^0+h\fy')-m_+(\fy^0)}{h}
 \pn=m_+'(\fy^0)(\fy').}
The linearity of $m_+'$ on $\fy'$ is clear from the definition of $\ck m_+$, while its boundedness follows from the Lipschitz property of $m_+$. 

To show the continuity of $m_+'$ for $\fy^0$, take any sequence $\fy^0_n\to\fy^0$ in $\B'_{\de_m}$, let $\vec u^0_n$ be the solution starting from 
$\vec u^0_n(0)=\Psi_{\s(0),c(0)}(m_+(\fy^0_n)g_++\fy^0_n)$ 
and let $\vec v^0_n:=\Psi_{\s,c}^{-1}(\vec u^0_n)$. The local wellposedness implies that $\vec v^0_n\to\vec v^0$ in $L^\I_\ta\cH\cap \St_\ta(0,S)$ for any $S\in(0,\I)$. 
Let $\eqref{Leq}_n$ be the equation obtained by replacing $v^0$ with $v^0_n$ in \eqref{Leq}. 

For any small $\z>0$, Lemma \ref{lem:instL} allows one to choose $S\gg 1$ such that the solution of \eqref{Leq} with $\la_+'(0)=1$ and $(\la_-'(0),\ga'(0))=0$ satisfies 
\EQ{
 \pt C_D \ig_I \la_+'(\ta)>\|(\la_-'(\ta),\ga'(\ta))\|_{\R\times\cH} +e^{k\ta/2}/\z}
for all $\ta>S/2$. On the other hand, for any $\fy'\in\B'_1$, the solution of \eqref{Leq} with $\la_+'(0)=m_+'(\fy^0)\fy'$ and $\la_-'(0)g_-+\ga'(0)=\fy'$ satisfies 
\EQ{
 \|(\la'(\ta),\ga'(\ta))\|_{\R^2\times\cH} \lec \ig_I^{-1}e^{C\ig_I\ta}}
for all $\ta\ge 0$. Since $S\gg 1$ and $k\gg\ig_I$, combining these two estimates yields that the solution of \eqref{Leq} with $|\la_+'(0)-m_+'(\fy^0)\fy'|>\z$ and $\la_-'(0)g_-+\ga'(0)=\fy'$ satisfies 
\EQ{ \label{la+' dom}
  \ig_I |\la_+'(\ta)| \gec \|(\la_-'(\ta),\ga'(\ta))\|_{\R\times\cH} }
for all $\ta>S/2$. 

The local uniform convergence of $v^0_n$ implies that the solution of $\eqref{Leq}_n$ with the same initial data also satisfies \eqref{la+' dom} around $\ta=S$ for large $n$. 
Moreover, since $(\ti\s(\vec u^0_n),\ti c(\vec u^0_n))\to(\s,c)$ uniformly on $[0,2S]$ as $n\to\I$, we have the same estimate \eqref{la+' dom} around $\ta=S$ also in the coordinate associated with the solution $\vec u^0_n$, for large $n$. Then it implies that 
\EQ{
 |m_+'(\fy^0_n)\fy'-m_+'(\fy^0)\fy'|\le\z,}
for all $\fy'\in B'_1$. Hence $m_+'$ is continuous $\B'_{\de_m}\to(\R\times\cH_\perp)^*$. 
This concludes the proof of Theorem \ref{thm:cmd0}.

\section{Extension of the manifold and the 9-set dynamics}
\subsection{Extension by the backward flow}
By the maximal evolution, we can extend $\M_0$ to an invariant manifold: 
\EQ{ \label{def M1}
 \M_1:=\Cu\{\vec u(I) \mid \vec u(0)\in\M_0,\ \vec u \in\Sol(I)\} \supset \M_0.}$\M_1$ also inherits from $\M_0$ the invariance for $\T$ and $\cS$. By the property of $\M_0$, those solutions are eventually trapped by the ground state, namely 
\EQ{
 \limsup_{t\nearrow T_+}d_W(\vec u(t))^2 \le 2(E(u)-E(W)) \lec\de_m^2,}
where $T_+$ is the maximal existence time of $u$. 

Conversely, if a solution $u$ satisfies the above condition and $E(\vec u)-E(W)\ll\de_m^2$, then its orbit is included in $\M_1$. 
This is because every solution getting close enough to the ground states is classified by the trichotomy of Theorem \ref{thm:cmd0}, whereas those solutions which never approach the ground states have been classified into the 4 sets of scattering to $0$ and blowup away from the ground states in \cite{CNW-nonrad}. One may wonder what happens if a solution stays around $d_W(\vec u)\sim\de_m$, but such behavior is precluded by the ejection lemma \cite[Lemma 3.2]{CNW-nonrad} (applied in both time directions) under the energy constraint $E(u)-E(W)\ll\de_m^2$. 

For each point $\fy\in\M_1$, there is a small neighborhood $O\ni\fy$ and $T>0$ such that the nonlinear flow $U_N(T)$ maps $O$ into a small neighborhood of a point on $\M_0$, where the trichotomy holds. Then $O\cap \M_1$ is mapped onto $U_N(T)(O)\cap\M_0$. Since $U_N(T)$ is smooth on $O$, it implies that $\M_1$ is also a $C^1$ manifold of codimension $1$. 

The one-pass theorem in \cite{CNW-nonrad} implies that every solution on $\M_1\cap\M_1^\da$ satisfies $d_W(\vec u(t))\lec\de_m$ all over its life, and so it is essentially the same as the center manifold $\M_0\cap\M_0^\da$, and in particular with codimension $2$. The rest of $\M_1$ is split into two parts, scattering to $0$ or blowup away from the ground states, in the negative time direction. Each set is non-empty and relatively open in $\M_0$. 

Therefore, we have all $3\times 3$ combinations of dynamics in $t>0$ and in $t<0$: (1) blowup away from the ground states, (2) trapping by the ground states (or by $\M_0$ for $t>0$ and by $\M_0^\da$ for $t<0$), and (3) scattering to $0$. It was already shown in \cite{CNW-nonrad} that the combinations of (1) and (3) have non-empty interior. Moreover, those 9 sets exhaust all possible dynamics in the region $E(u)-E(W)\ll\min(\de_m,\e_*)^2$. Thus we obtain Theorem \ref{main thm1}. 

\medskip

Before going to the next step using the Lorentz transform, it is convenient to consider the space-time maximal extension of each solution of (CW). 

\subsection{Space-time extension and restriction}
To solve the equation locally in and out of light cones, and in more general space-time sets, we introduce restricted energy semi-norms. Let $\sB_d$ be the totality of Borel sets in $\R^d$. For any $B\in\sB_d$ and any $a\ge 0$, we define two sets $B_{\pm a}\in\sB_d$ by 
\EQ{ \label{def B+-}
 \pt B_{+a}:=\{x\in\R^d \mid \exists y\in B,\ |x-y|\le a\},
 \pr B_{-a}:=\{x\in\R^d \mid |x-y|\le a \implies y\in B\}.}
It is clear that for any $a,b,t\ge 0$, we have 
\EQ{
 \pt (B_{+a})_{+b}=B_{+(a+b)}, \pq (B_{-a})_{-b}=B_{-(a+b)}, 
 \pr (B_{-a})_{+a} \subset B \subset (B_{+a})_{-a},
 \pq (B_{+a})^\cmpl=(B^\cmpl)_{-a}.}
For any $I\subset\R$ and any $F:\R\to\R$, we define $B_{\pm F}(I)\in\sB_{1+d}$ by
\EQ{ \label{def B+-I}
 B_{\pm F}(I) := \{(t,x)\in I\times \R^d \mid x\in B_{\pm F(t)}\}.}
For any $B\in\sB_d$, let $V(B)\subset\cH$ be the closed subspace defined by 
\EQ{
 V(B):=\{\fy\in\cH\mid \fy(x)=0\ a.e.x\in B\},}
and then define the restriction of $\cH$ onto $B$ by
\EQ{ \label{def cHR}
 \tHR{B}:=\cH/{V(B)} \simeq V(B)^\perp,}
where $\simeq$ means the isometry, with the quotient norm
\EQ{ 
 \|\fy\|_{\cHR{B}}:=\inf\{\|\psi\|_\cH \mid \psi=\fy \ton{B}\}\pq(\fy\in\cH).} 
Henceforth, we denote for brevity 
\EQ{
 \fy=\psi\ton{B} \overset{\text{def}}{\iff} \fy-\psi\in V(B).}
We also use the more explicit semi-norms for $\fy\in\cH$: 
\EQ{ \label{def ticH}
 \pt \|\fy\|_{\cH(B)}^2:=\int_{B}[|\na u_1|^2+|u_2|^2]dx, 
 \pq \|\fy\|_{\ti\cH(B)}:=\|\fy\|_{\cH(B)}+\|\fy_1\|_{L^{2^*}(B)}.}
All of these three semi-norms are increasing for $B$ and invariant for $\T,\cS$, namely
\EQ{
 \|\fy\|_{X(B_1)} \le \|\fy\|_{X(B_1\cup B_2)}, 
 \pq \|\T^c\cS^s\fy\|_{X(B)}=\|\fy\|_{X(e^{-\s}B+c)},}
for $X=\cH,\ti\cH$, and $\tHR{\ }$. We have, uniformly for $B$, 
\EQ{
 \|\fy\|_{\cH(B)} \le \|\fy\|_{\cHR{B}}, \pq \|\fy\|_{\ti\cH(B)}\lec\|\fy\|_{\cHR{B}}.}
We may have the reverse inequalities when $B$ is smooth. In particular, we have 
\EQ{
 \pt \|\fy\|_{\cHR{\{|x-c|<R\}}}\sim\|\fy\|_{\ti\cH(|x-c|<R)},
 \pr \|\fy\|_{\cHR{\{|x-c|>R\}}}\sim\|\fy\|_{\ti\cH(|x-c|>R)}\sim\|\fy\|_{\cH(|x-c|>R)},}
uniformly for $c$ and $R$, where the open region can be replaced with the closure. The extension operator $X_B:\cH\to V(B)^\perp\subset\cH$ is nothing but the orthogonal projection to $V(B)^\perp$, such that we have 
\EQ{ \label{def XB}
 \|X_B\fy\|_\cH = \|\fy\|_{\cHR{B}}, \pq X_B\fy=\fy \ton{B}.}
Restriction of energy-type functionals is denoted as follows 
\EQ{ \label{def EB}
 E_B(\fy):=\frac{\|\fy\|_{\cH(B)}^2}{2}-\frac{\|\fy_1\|_{L^{2^*}(B)}^{2^*}}{2^*},
 \pq K_B(\fy):=\|\na\fy\|_{L^2(B)}^2-\|\fy\|_{L^{2^*}(B)}^{2^*}.}

The finite propagation speed implies that if a solution $u$ of (CW) satisfies 
\EQ{
 \vec u(0)=\psi\in\tHR{B}\ \ton{B},}
then $u$ is uniquely determined on $B_{-|t|}$ by $\psi$. More precisely, if $\vec u^0,\vec u^1\in C([0,T];\cH)$ satisfy $\vec u^0(0)=\vec u^1(0)$ on $B$, then 
\EQ{
 0\le \forall t\le T,\ a.e.x\in B_{-|t|},\ \vec u_0(t,x)=\vec u_1(t,x).}
By the Strichartz estimate, there is $C>0$ such that if $C\|\psi\|_{\cHR{B}}\le\e_S$ then there is a free solution $v$ satisfying $\vec v(0)=\psi$ on $B$ and $\|v\|_{\St(\R)}\le \e_S$, and so is $u\in\Sol(\R)$ satisfying $\vec u(0)=\psi$ on $B$ and $\|u\|_{\St(\R)}\lec \e_S$, which is unique on $B_{-|t|}(\R)$. 

Now we introduce the space-time maximal extension of a solution of (CW). 
For any $\fy\in\cH$, $c\in\R^d$ and $R>0$, consider the local solution in the light cone $\K_{c,R}:=\{|x-c|+t<R,t\ge 0\}$ with the initial data $\vec u(0)=\fy$ on $|x-c|<R$. Let $t_+(\fy,c)$ be the supremum of such $R$ that there is a unique solution $u$ in $\K_{c,R}$ satisfying $\|u\|_{L^{q_m}(\K_{c,R})}<\I$. 
The uniqueness in cones implies that we have a unique solution $u$ in the space-time region 
\EQ{
 \{(t,x)\in\R^{1+d} \mid 0\le t<t_+(\fy,x)\},}
as well as the Lipschitz continuity 
\EQ{
 |t_+(\fy,x)-t_+(\fy,y)| \le |x-y|.} 
We also write for any strong solution $u$ (either before or after the above space-time maximal extension), 
\EQ{
 t_+(u,x):=t_+(\vec u(0),x).} 
The maximal existence time is then given by 
\EQ{ 
 T_+(u):=\inf_{x\in\R^d}t_+(u,x).} 
The small data theory in interior and exterior cones implies that for any $\fy\in\cH$, there are $a(\fy),b(\fy)>0$ such that 
\EQ{
 t_+(\fy,x) \ge \max(a(\fy),|x|-b(\fy)).}
The definition of $t_+$ implies that 
\EQ{ \label{def t*}
 \|u\|_{L^{q_m}(\K_{c,t_+(u,c)})}=\lim_{R\nearrow t_+(u,c)} \|u\|_{L^{q_m}(\K_{c,R})}=\I,}
and so the small data theory implies 
\EQ{
 \liminf_{t\nearrow t_+(u,c)}\|\vec u(t)\|_{\ti\cH(|x-c|<t_+(u,c)-t)}\gec\e_S.}
In particular, the number of first blow-up points is bounded in the case of type-II
\EQ{ \label{bup points}
 \#\{c\in\R^d \mid t_+(u,c)=T_+(u)\} \lec \liminf_{t\nearrow T_+(u)}\sqrt{\|\vec u(t)\|_\cH/\e_S}.}

Similarly we can define $t_-(\fy,c)<0$ to be the maximal extension in the negative direction, and thus a unique solution $u$ in the maximal space-time domain 
\EQ{ \label{def D}
 \dom(u):=\dom(\fy):=\{(t,x)\in\R^{1+d}\mid t_-(\fy,x)<t<t_+(\fy,x)\},}
satisfying for some $a(\fy),b(\fy)>0$ and for all $x,y\in\R^d$, 
\EQ{ \label{prop tpm}
 \pm t_\pm(\fy,x) \ge \max(a(\fy),|x|-b(\fy)),
 \pq |t_\pm(\fy,x)-t_\pm(\fy,y)| \le |x-y|.}
Since the Lorentz transforms preserve light cones as well as the measure and the topology of $\R^{1+d}$, the property \eqref{def t*} of $t_+$ is also preserved. Hence, each strong solution $u$ defined on its maximal space-time domain $\dom(u)$ is transformed by any Lorentz transform into another solution defined on the maximal domain. This process can produce a solution with no Cauchy time slice, namely $\inf t^+<\sup t^-$, but we ignore such solutions, in order to keep the dynamical viewpoint in terms of the Cauchy problem or the flow in $\cH$. In other words, the Lorentz transforms should be restricted to the range where $\inf t^+<\sup t^-$ is kept. 

For the blow-up solutions on the center-stable manifold, we have 
\begin{lem}
Let $0<T<\I$ and $\vec u\in\Sol([0,T))$ satisfy
\EQ{
 \vec u(0)\in\M_0, \pq \|d_W(\vec u)\|_{L^\I_t(0,T)}\le\de\lec\de_m.} 
Then there exists $c_*\in\R^d$ and $\e>0$ such that 
\EQ{ \label{bcurv on M}
 \pt e^{-\ti\s(\vec u(t))}+|\ti c(\vec u(t))-c_*| \lec \de|t-T|, 
 \pr t_+(\vec u(0),x)=T+|x-c_*|, \pq t_-(\vec u(0),x)<T-|x-c_*|-\e.}
\end{lem}
\begin{proof}
Let $(\s(t),c(t)):=(\ti\s(\vec u(t)),\ti c(\vec u(t))$. 
First we show $\s(t)\to\I$ as $t\to T-0$. If not, there exist a sequence $t_n\nearrow T$ with $\sup_n\s(t_n)<\I$ and $R>0$ such that 
\EQ{
 \sup_n \sup_{c\in\R^d}\|\vec u(t_n)\|_{\ti\cH(|x-c|<R)} \lec \de \ll \e_S,}
which ensures solvability in the cone $|x-c|+|t-t_n|<R$ for all $c\in\R^d$ and all $t_n$, thereby extending the solution $u$ beyond $T$, a contradiction. Hence $\s(t)\to\I$. Then the modulation equation \eqref{eq sct} implies convergence $c(t)\to\exists c_*\in\R^d$ as well as the first estimate of \eqref{bcurv on M}. Since $\|\vec u(t)-\T^c\cS^\s\vec W\|_\cH\sim d_W(\vec u(t))<\de$, that behavior of $(\s,c)$ implies that for any $R>0$, 
\EQ{
 \limsup_{t\nearrow T}\|\vec u(t)\|_{\cH(|x-c_*|>R)}\lec\de \ll \e_S,}
which ensures solvability in the exterior cone $|x-c_*|+|t-T'|>R$ for all $T'<T$ close to $T$, and so $t_+(\vec u(0))\ge T'+|x-c_*|-R$ and $t_-(\vec u(0))\le T'-|x-c_*|+R$. Letting $R\searrow 0$ and $T'\nearrow T$, we obtain 
\EQ{
 t_+\ge T+|x-c_*|, \pq t_- \le T-|x-c_*|.}
It can not be better for $t_+$ as $u$ blows up at $(T,c_*)$, so we obtain the identity in \eqref{bcurv on M}. The finite propagation implies $\|\vec u(t)\|_{\cH(|x-c_*|>R+|T-t|)}\lec\de$, so that we can solve in a slightly larger exterior cone starting from $t=T'\in(0,T)$. Thus we obtain the last estimate in \eqref{bcurv on M}. 
\end{proof}

\subsection{Extension by the Lorentz transform}
Next we use the Lorentz transforms to extend the manifold, so that it can include all the ground state solitons $\Soliton(W)$. 

Let $\vec u\in\Sol([0,T))$, $\vec u(0)\in\M_0$ and $T<\I$. 
The estimate \eqref{bcurv on M} on the blowup surface implies existence of $T'\in(0,T)$ such that $\{|x-c_*|\ge|t-T'|\}\subset\dom(u)$. Then for any Lorentz transform centered at $(T',c_*)$, the solution $u$ is transformed into another solution on the maximal domain containing the time slice $\{t=T'\}$. In the case where $\vec u(0)\in\M_0$ yields $\vec u\in\Sol([0,\I))$, Lorentz transformed solutions are also defined for all large $t$, because of \eqref{prop tpm}.

Let $u$ be any solution with $\vec u(0)\in\M_0$ defined on $\dom(u)$, and let $w$ be any Lorentz transform of $u$. The above argument implies that $w$ is defined on some time slice. 
Let $\M_2$ be the totality of the maximal orbit of all such solutions $w$. Then it is invariant by the flow, $\T$ and $\cS$, satisfying
\EQ{ \label{def M2}
 \M_2 \supset \M_1 \cup \Soliton(W).}
Every solution on $\M_2$ is Lorentz transformed to another solution on $\M_1$. $\M_2$ is Lorentz invariant in the following sense: Let $u$ be a solution on $\M_2$, extended to $\dom(u)$. Then any Lorentz transform of $u$ has non-trivial maximal existence interval, on which the transformed solution belongs to $\M_2$. 

To see that $\M_2$ is locally $C^1$ diffeo to $\M_1$, it suffices to see that the Lorentz transform gives a local $C^1$ mapping around any solution. 
Let $\vec u\in\Sol([-T,T])$ for some $T>0$, then the local wellposedness yields a neighborhood $O\ni\vec u(0)$ and $R>0$ such that for any $\fy\in O$ we have 
\EQ{
 \dom(\fy)\supset X:=\{(t,x)\in\R^{1+d}\mid |t|<\max(T/2,|x|-R)\}.}
Then there is a neighborhood $U\ni 1$ in the Lorentz group such that every transform in $U$ maps the region $X$ to a set containing $[-T/4,T/4]\times\R^d$. 
Since the space rotation plays no role, we may restrict to those transforms defined on $(t,x_1)\in\R^2$, which can be parametrized as 
\EQ{
 u^\te(t,y,z)=u(ct+sy,cy+st,z), \pq c:=\cosh\te,\ s:=\sinh\te, \pq \te\in\R,}
where $x=(y,z)\in\R^1\times\R^{d-1}$. Then there is $\Te>0$ such that for any $\vec u(0)\in O$ and for any $\te\in(-\Te,\Te)$, we have $\dom(u^\te)\supset[-T/4,T/4]\times\R^d$. This defines a mapping 
\EQ{
 \sL_\te:O\ni\vec u(0)\mapsto \vec u^\te(0)\in\cH} 
for each $\te\in(-\Te,\Te)$. 
The continuity of $\sL_\te$ can be seen by the linear energy identity. For any smooth function $u$ defined on $X$, we have by the divergence theorem
\EQ{
 \pt E^F(\vec u^\te(0))= cE^F(\vec u(0))+sP_1(\vec u(0))+\int_{\frac{sy}{ct}>1} (\ddot u-\De u)\frac{t}{|t|}(cu_t+su_y)dxdt,
 \pr P_1(\vec u^\te(0))= sE^F(\vec u(0))+cP_1(\vec u(0))+\int_{\frac{sy}{ct}>1} (\ddot u-\De u)\frac{t}{|t|}(su_t+cu_y)dxdt,}
where $E^F(\fy):=\|\fy\|_\cH^2/2$ denotes the free energy. 
Applying this to the difference of two solutions $u,w$ starting from $O$ yields 
\EQ{
 \|\vec u^\te(0)-\vec w^\te(0)\|_\cH^2
 \pt\lec \|\vec u(0)-\vec w(0)\|_\cH^2 
 \prq+ \|f'(u)-f'(w)\|_{L^1_tL^2_x(sy/(ct)<1)}\|\vec u-\vec w\|_{L^\I_t\cH_x(sy/(ct)<1)},}
where the implicit constants depend on $\Te$. The standard perturbation argument implies that the last term is also bounded by $\|\vec u(0)-\vec w(0)\|_\cH^2\ll 1$. The existence and continuity of the derivative of $\sL_\te$ is shown similarly by applying the energy estimate to the linearized equation 
\EQ{
 (\p_t^2-\De)u'=f''(u)u', \pq \vec u'(0)\in\cH.} 
Since $(\sL_\te)^{-1}=\sL_{-\te}$ is obvious, we thus conclude that $\sL_\te$ is a local $C^1$ diffeo on $O$. 
Therefore $\M_2$ is also a $C^1$ manifold with codimension $1$ in $\cH$. It is clear from the construction that all these manifolds $\M_0\subset\M_1\subset\M_2$ are connected. 

For the trichotomy around $\M_2$, it is obvious from the energy estimate that every scattering (to $0$ as $t\to\I$) solution is transformed into another such solution by any Lorentz transform. The solutions in the other part of the neighborhood blow up away from a (much) bigger neighborhood, which is generated from the $\de_X$ neighborhood of $\Static(W)$ by the above extensions. Thus we obtain the 9-set dynamics classification in the region 
\EQ{
 E(u) < \sqrt{(E(W)+\e^2)^2+|P(u)|^2}.}
Reduction of this region to $E(u)<E(W)+\e^2$ is done as in \cite{CNW-nonrad} by the Lorentz transform and the identities for (CW)
\EQ{
 \pt E(\vec u^\te)= cE(\vec u)+sP_1(\vec u),
 \pq P_1(\vec u^\te)= sE(\vec u)+cP_1(\vec u).}
Indeed, if $E(u)<|P(u)|$ then we can transform it (in some space-time region) to another solution with negative energy, which has to blow up in both time directions by the classical argument of Levine \cite{Lev}, or more precisely by \cite{KM}. Hence the solution before the transform should also blow up in both directions. 

If $E(u)=|P(u)|$ and it is global for $t>0$, then there is a sequence of solutions $\vec u^n\in\Sol([0,\I))$ given by Lorentz transforms such that $E(u_n)\to 0$. 
Then the classical argument of Payne-Sattinger \cite{PS} implies that $K(u_n)\ge 0$ as soon as $E(u_n)<E(W)$, and so $E(u_n)\sim\|\vec u_n\|_{L^\I_t\cH}^2\to 0$. 
The small data scattering implies that $\|u_n\|_{L^{q_m}(\R^{1+d})}\lec E(u_n)^{1/2}\to 0$, but since the Lorentz transform is measure preserving on $\R^{1+d}$, it implies that $\|u\|_{L^{q_m}(\R^{1+d})}=0$. 
In short, all the solutions with $E(u)\le |P(u)|$ blow up in both time directions except for the trivial solution $0$. 

If $E(u)>|P(u)|$, then we can transform it to another solution $\ti u$ with $P(\ti u)=0$ and $E(\ti u)<E(W)+\e^2$, and so $\ti u$ should either scatter to $0$ as $t\to\I$, blow up away from the ground state in the positive time direction, or live on $\M_1$. Each of those properties is transferred back to the original solution $u$. Note that if $\dom(\ti u)$ contains no time slice then the original solution $u$ must blow up in both time directions. 
Thus we complete the 9-set dynamics classification slightly above the ground states, and the proof of Theorem \ref{main thm2}. 

\vspace{36pt}

{\large \bf Part II: Large radiation}\\

The goal in the rest of paper is to extend the center-stable manifold to the entire energy space $\cH$, together with the dynamics around it, by a simple argument which allows one to reduce the problem to $\M_0$ in the region $E(u)<E(W)+\e^2$, using the asymptotic Huygens principle together with the finite speed of propagation. 

\section{Detaching the radiation}
For any $B\in\sB_d$ and any $T\in(0,\I]$, we define a semi-norm $\Ra_B^T$ in $\cH$ by
\EQ{ \label{def Ra}
 \|\fy\|_{\Ra_B^T}:=\inf\{\|U(t)\psi \|_{L^\I_t(0,T;\cHR{B_{+t}})\cap \St(0,T)} 
 \mid \psi=\fy \ton{B^\cmpl} \}.}
Smallness in $\Ra_B^T$ will imply that we can detach the exterior component using the wave starting from $\psi$ which is out-going dispersive in the sense of the energy on the interior cone $B_{+t}$, and also the Strichartz norm on $\R^d$, both for $0\le t\le T$. 

The lower semi-continuity of the norms implies that the infimum in defining $\Ra_B^T$ is achieved by some $\psi\in\cH$ such that $\psi=\fy$ on $B^\cmpl$ and 
\EQ{ \label{Ra bd Hex}
 \|\fy\|_{\Ra_B^T}=\|U(t)\psi \|_{L^\I_t(0,T;\cHR{B_{+t}})\cap \St(0,T)} \lec \|\psi\|_\cH \sim \|\fy\|_{\cHR{B^\cmpl}}.}
For the last equivalence, $\ge$ is by definition of $\tHR{B^\cmpl}$, while $\lec$ follows from 
\EQ{
 \|\psi\|_\cH \le \|\psi\|_{\cHR{B}}+\|\psi\|_{\cHR{B^\cmpl}}
 \le \|\fy\|_{\Ra_B^T}+\|\fy\|_{\cHR{B^\cmpl}} \lec \|\fy\|_{\cHR{B^\cmpl}}.}

The following laws of order are trivial by definition 
\EQ{
 \pt \ta\ge 0 \implies \|\fy\|_{\Ra_B^{T}} \le \|\fy\|_{\Ra_B^{T+\ta}},
  \pq \|U(\ta)\fy\|_{\Ra_{B_{+\ta}}^{T-\ta}} \le \|\fy\|_{\Ra_B^T},}
as is the invariance 
$\|\T^c\cS^\s\fy\|_{\Ra_B^T} = \|\fy\|_{\Ra_{e^\s B+c}^{e^\s T}}$, but it is not invariant under the time inversion $\fy\mapsto\fy^\da$.  
The space-time continuity of the norms implies 
\EQ{ \label{loc Ra}
 \lim_{R,T\to+0}\|\fy\|_{\Ra_{|x-c|<R}^T} =0.}
Also note the trivial identity $\|\fy\|_{\Ra_{\R^d}^T}=0$.

The following ``asymptotic Huygens principle" plays a crucial role in using the dispersive property of (CW) in the above function space.
\begin{lem} \label{lem:Huy}
For any free solution $\vec v\in C(\R;\cH)$ and bounded $B\in\sB_d$ we have 
\EQ{
 \lim_{T\to\I}\sup_{t>0}\|\vec v(T+t)\|_{\cHR{B_{+t}}}=0.}
\end{lem}
\begin{proof}
Since $B$ is included in some ball, it suffices to prove 
\EQ{
 \lim_{T\to\I}\sup_{t>T}\|\vec v(t)\|_{\ti\cH(|x|<t-T)}=0.}
Since the statement is obviously stable in the energy norm, we may restrict the initial data to a dense set, say $C_0^\I(\R^d)$. Multiplying the equation with $(t^2+r^2)\dot v + 2trv_r+(d-1)tv$, we obtain conservation of the conformal energy 
\EQ{
 \int_{\R^d}|tv_t+rv_r+(d-1)v|^2+|rv_t+tv_r|^2+(d-1)|v|^2+(t^2+r^2)|\na^\perp v|^2dx,}where $\na^\perp v:=\na v-xv_r$ denotes the derivative in the angular directions. Since in the region $|x|<t-T$ 
\EQ{
 \pt\left|\mat{v_t \\ v_r}\right|
 \pn=\left|\frac{1}{t^2-r^2}\mat{t & -r \\ -r & t}\mat{tv_t+rv_r \\ rv_t+tv_r}\right|
 \pn\le\frac{1}{T}\left|\mat{tv_t+rv_r \\ rv_t+tv_r}\right|,}
the $L^2$ norm of the left tends to $0$ uniformly as $T\to\I$, as well as those of $v/(t-T)$ and $\na^\perp v$, while $\|v\|_{L^{2^*}(\R^d)}\to 0$ by the free dispersive decay.
\end{proof}

The asymptotic Huygens principle implies the following decay of $\Ra_B^\I$: For any $\fy\in\cH$ and any bounded $B\in\sB_d$,   
\EQ{ \label{disp Ra}
 \lim_{\ta\to\I}\|U(\ta)\fy\|_{\Ra_B^\I}=0.}
In other words, every free solution in $\cH$ will eventually gets into any small ball of $\Ra_B^\I$ around $0$, as well as every scattering solution of (CW) in $\cH$. Also, when the solution is around the ground states with large dispersive remainder, we can take the semi-norm $\Ra_B^T$ small by the following.
\begin{lem} \label{lem:est Ra}
There is $\e_D>0$ with the following property. Let $\vec u\in\Sol(I)$, $(\s,c):I\to\R^{1+d}$, $\vec F(t)=U(t)\vec F(0)\in\cH$ and $R\in C(I;\cH)$ satisfy
\EQ{
 \pt I\ni\forall t,\ \vec u(t)=\T^{c(t)}\cS^{\s(t)}\vec W+\vec F(t)+R(t), 
 \pr \|F\|_{\St(I)}+\|R\|_{L^\I(I;\cH)}=:\ra\le\e_D,}
for some interval $I$. Then for any $t_0,t_1\in I$ we have 
\EQ{ \label{est sc}
 |e^{-\s(t_0)}-e^{-\s(t_1)}|+|c(t_0)-c(t_1)|\lec \ra(|t_0-t_1|+e^{-\s(t_0)}).}
If $I=[0,T)$, then there are $t_0 \in(0,T)$ and $B:[t_0,T)\to\sB_d$ such that 
\EQ{ \label{est Ra}
 \sup_{t_0 \le t<T} \|\T^{c(t)}\cS^{\s(t)}\vec W\|_{\cHR{B(t)^\cmpl}} + \|\vec F(t)\|_{\Ra_{B(t)}^{T-t}} \lec \ra^{d/2-1}.}
Furthermore, if $T<\I$, then $\s(t)\to\I$ and $c(t)\to\exists c_*\in\R^d$ as $t\to T-0$. 
\end{lem}
\begin{proof}
Let $\psi(t):=\T^{c(t)}\cS^{\s(t)}\vec W$, $\vec v(t):=U(t)(\vec F(0)+R(0))$, and $\vec w(t):=\vec u(t)-\psi(0)$. Then we have 
\EQ{
 \square w=f'(\psi_1(0)+w)-f'(\psi_1(0)), \pq \vec w(0)=\vec v(0),}
hence by Strichartz 
\EQ{
  \|\vec w-\vec v\|_{L^\I\cH\cap\St(0,S)}
 \pt\lec \|w\|_{\St_s(0,S)}(\|\psi(0)\|_{\St_m(0,S)}+\|w\|_{\St_m(0,S)})^{2^*-2}
 \pr\lec \|w\|_{\St_s(0,S)}(|e^{\s(0)}S|^{1/q_m}+\|w\|_{\St_m(0,S)})^{2^*-2},}
and $\|\vec v\|_{\St(0,S)} \lec \|\vec F\|_{\St(0,S)}+\|R(0)\|_\cH\le\ra$, 
for any $S\in(0,T)$. Choosing $\e_D\ll 1$, we deduce that as long as $0<e^{\s(0)}S<\e_D$ 
\EQ{
 \|\vec w-\vec v\|_{L^\I\cH\cap\St(0,S)} \ll \|\vec v\|_{\St(0,S)} \lec \ra.}Since $\vec w(t)-\vec v(t)=\psi(t)-\psi(0)+R(t)-U(t)R(0)$, it implies that for $0<t<e^{-\s(0)}\e_D$, 
\EQ{
 \ra\pt\gec\|\vec w(t)-\vec v(t)\|_\cH+\|R(t)\|_\cH+\|R(0)\|_\cH
 \pr\gec\|\psi(t)-\psi(0)\|_\cH \sim |\s(t)-\s(0)|+e^{\s(0)}|c(t)-c(0)|.}

Now define a time sequence inductively by $t_0=0$ and $t_{j+1}=t_j+e^{-\s(t_j)}\e_D$. Then applying the above argument from $t_j$ yields 
\EQ{
 |\s(t)-\s(t_j)|+e^{\s(t_j)}|c(t)-c(t_j)| \lec \ra}
for $t_j\le t\le t_j+e^{-\s(t_j)}\e_D$, and induction on $j$ yields the desired estimate \eqref{est sc}. 

If $I=[0,\I)$, let $B(t):=\{|x-c(0)|<e^{-\s(0)}+t/2\}$. Lemma \ref{lem:Huy} implies 
\EQ{
 \|\vec F(s)\|_{\Ra_{B(s)}^\I} \lec \|\vec F\|_{L^\I_t(s,\I;\cHR{B(s)_{+t}})} + \|F\|_{\St(s,\I)} \to 0}
as $s\to\I$, while $\|\psi(s)\|_{\cHR{B(s)^\cmpl}}\lec\ra^{d/2-1}$ for large $s$, by \eqref{est sc} together with the explicit form of $W$. Thus we obtain \eqref{est Ra}.

If $I=[0,T)$ and $T<\I$, then $\s(t)\to\I$ as $t\nearrow T$, since otherwise there is a sequence $t_n\nearrow T$ with $\sup_n\s(t_n)<\I$ and $\de>0$ such that 
\EQ{
 \sup_n \sup_{c\in\R^d}\|\vec u(t_n)\|_{\ti\cH(|x-c|<\de)} \lec \ra.}
Choosing $\e_D\ll\e_S$, this ensures solvability in the cone $|x-c|+|t-t_n|<\de$ for all $c\in\R^d$ and all $t_n$, thereby extending the solution beyond $T$, a contradiction. Hence $\s(t)\to\I$. Then \eqref{est sc} implies the convergence $c(t)\to\exists c_*\in\R^d$. Let $B:=\{|x-c_*|<R\}$. Then as $s\nearrow T$ we have 
\EQ{
 \pt \|\vec F(s)\|_{\Ra_B^{T-s}} \le \|\vec F\|_{L^\I(s,T;\cHR{B_{T-s})}}+\|\vec F\|_{\St(s,T)} \to \|\vec F(T)\|_{\cHR{B}},}
and $\|\psi(s)\|_{\cHR{B^\cmpl}} \to 0$. Choosing $R$ small enough yields \eqref{est Ra}.
\end{proof}

The following Sobolev-type inequality implies that $\Ra_B^\I$ controls $L^{2^*}$, which is a notable difference from $\Ra_B^T$ with $T<\I$. 
\begin{lem}[Time-Sobolev for the Strichartz norms]
Let $\N\ni d\ge 2$, $\Z\ni k\ge 0$ and $2\le q\le r\le\I$. 
Then for any free solution $v$ we have 
\EQ{ 
 \|\p_t^k v\|_{L^\I_t(0,\I;\dot B^{\s-1/q-k}_{r,r})} \lec \|v\|_{L^q_t(0,\I;\dot B^\s_{r,r})}.}
In particular, if $d\ge 3$, $1/q+d/r-\s = d/2-1$ and $\s>1/q$, then  
\EQ{ \label{to L6}
 \|\p_t^k v\|_{L^\I_t(0,\I;|\na|^k L^{2^*})} \lec \|v\|_{L^q_t(0,\I;\dot B^\s_{r,2})}.}
\end{lem}
The last inequality applies to any Strichartz norm of the $\dot H^1$ scaling with the condition $1/\s<q\le r$, in particular to $\St_s=L^{q_s}\dot B^{1/2}_{q_s,2}$. Hence we have  
\EQ{ \label{St L6}
 \|\vec v\|_{L^\I_t(0,\I;(1\oplus |\na|)L^{2^*})} \lec \|\vec v\|_{\St_s(0,\I)}.}
\begin{proof}
Let $v=\sum_{j\in\Z}v_j$ be a standard Littlewood-Paley decomposition in $x\in\R^d$ with $\supp\F v_j\subset\{2^{j-1}<|\x|<2^{j+1}\}$. The wave equation and the property of the L-P decomposition imply
\EQ{
 \pt \|\ddot v_j\|_{L^{r}_xL^{q}_t}=\|\De v_j\|_{L^{r}_xL^{q}_t}
 \sim 2^{2j}\|v_j\|_{L^{r}_xL^{q}_t},}
on $(0,\I)\times\R^d$. Then by the elementary interpolation inequalities
\EQ{ \label{elem intp}
 \pt\|\dot u\|_{L^q(0,\I)} \lec \|u\|_{L^q(0,\I)}^{1/2}\|\ddot u\|_{L^q(0,\I)}^{1/2},
 \pq\|u\|_{L^\I(0,\I)} \lec \|u\|_{L^q(0,\I)}^{1-1/q}\|\dot u\|_{L^q(0,\I)}^{1/q},}
we obtain 
\EQ{ \label{est vj}
 \pt \|v_j\|_{L^{r}_xL^{\I}_t} \lec \|v_j\|_{L^{r}_xL^{q}_t}^{1-1/(2q)}\|\ddot v_j\|_{L^{r}_xL^{q}_t}^{1/(2q)} \sim 2^{j/q}\|v_j\|_{L^{r}_xL^{q}_t},
 \pr \|\dot v_j\|_{L^r_xL^\I_t} \lec \|v_j\|_{L^r_xL^q_t}^{1/2-1/(2q)}\|\ddot v_j\|_{L^r_xL^q_t}^{1/2+1/(2q)} \sim 2^{j(1/q+1)}\|v_j\|_{L^{r}_xL^{q}_t}.  }
Using Minkowski and Littlewood-Paley yields
\EQ{
 \|v\|_{L^\I_t \dot B^{\s-1/q}_{r,r}} \pt\le \|2^{j(\s-1/q)}v_j\|_{\ell^r_j L^r_x L^\I_t} \lec \|2^{j\s} v_j\|_{\ell^r_j L^r_x L^q_t}
 \pn\le \|v\|_{L^q_t \dot B^\s_{r,r}}.}
The estimate on the time derivatives follows in the same way. \eqref{to L6} follows from the standard Sobolev embedding for the Besov space, $\dot B^{\s-1/q}_{r,2}\subset \dot B^{\s-1/q}_{r,r}\subset L^{2^*}$. 
 
Finally, \eqref{elem intp} can be proved as follows. 
By the density argument, we may assume that $v$ is nonzero, real analytic, and exponentially decaying. Then $\dot v$ has at most countable number of zeros $0\le z_1<z_2<\cdots$ with no accumulation. For each $z_k<z_{k+1}$ we have, denoting $[v]^q:=|v|^{q-1}v$,   
\EQ{
 \|\dot v\|_{L^1(z_k,z_{k+1})}^q \pt= |v(z_k)-v(z_{k+1})|^q \le 2^{q-1} |[v(z_k)]^q-[v(z_{k+1})]^q|
 \pr\le 2^{q-1}\int_{z_k}^{z_{k+1}} q|v|^{q-1}|\dot v| dt \le 2^{q-1}q\|v\|_{L^q(z_k,z_{k+1})}^{q-1}\|\dot v\|_{L^q(z_k,z_{k+1})},}
and similarly,
\EQ{
 \|\dot v\|_{L^\I(z_k,z_{k+1})}^q \le \re\int_{z_k}^{z_{k+1}}q[\dot v]^{q-1}\overline{\ddot v}dt \le q\|\dot v\|_{L^q(z_k,z_{k+1})}^{q-1}\|\ddot v\|_{L^q(z_k,z_{k+1})}.}
The same argument yields the second inequality of \eqref{elem intp}. Interpolating the above two estimates by H\"older  
\EQ{
 \|\dot v\|_{L^q(z_k,z_{k+1})} \le 2^{(q-1)/q^2}q^{1/q} \|v\|_{L^q(z_k,z_{k+1})}^{(q-1)/q^2}\|\dot v\|_{L^q(z_k,z_{k+1})}^{1/q^2+(1-1/q)^2}\|\ddot v\|_{L^q(z_k,z_{k+1})}^{(q-1)/q^2},}
and so $\|\dot v\|_{L^q(z_k,z_{k+1})} \le [2q^{1-1/q} \|v\|_{L^q(z_k,z_{k+1})}\|\ddot v\|_{L^q(z_k,z_{k+1})}]^{1/2}$. Taking the $\ell^q$ sum over $k$, we obtain the first inequality of \eqref{elem intp}. 
\end{proof}
\begin{rem}
It is obvious from the homogeneous nature that the above lemma fails on any bounded set in $\R$ for any $q<\I$. If such an inequality would hold, then it must be uniform for the rescaling $u(t,x)\mapsto u(\la t,\la x)$, but the $L^q_t(I)$ norm decays as $I$ shrinks to a point $t_0\in\R$, while the $L^\I_t(I)$ norm converges to the value at $t_0$. 
\end{rem}

The following lemma allows us to detach exterior radiation which is small in $\Ra_R^T$ from any solution of (CW). 

\begin{lem}[Detaching lemma] \label{lem:detach}
Let $B\in\sB_d$, $\ti T \ge T>0$, and $u\in\Sol([0,T])\cup\Sol([0,T))$ satisfy $\|\vec u(0)\|_{\Ra_B^{\ti T}}=\ra<\e_S$. Then 

{\rm (I)} There are a free solution $v$, and two strong solutions $u\xt$ and $u\dt$ of (CW), defined on $[0,\ti T)$ and on $[0,T)$ respectively, satisfying 
\EQ{ \label{v ux ud}
 \pt \vec v(0)=\vec u(0) \ton B^\cmpl,\ \|\vec v\|_{L^\I_t(0,\ti T;\cHR{B_{+t}})\cap\St(0,\ti T)} \lec \ra,
 \pr \vec u\xt(t)=\vec u(t) \ton (B_{+t})^\cmpl, \pq \|\vec u\xt-\vec v\|_{(L^\I\cH\cap\St)(0,\ti T)} \lec \ra^{2^*-1}, 
 \pr \vec u\dt(t)=\vec u(t) \ton B_{+t}, \pq \|\vec u-\vec u\dt-\vec v\|_{(L^\I\cH\cap\St)(0,T)} \lec \ra,}
and $\|\vec v(0)\|_\cH\lec\|\vec u(0)\|_{\cHR{B^\cmpl}}$. More precisely, there is $w\in C([0,\ti T);\cH)$ such that $\vec u=\vec u\dt+w$ for $0\le t<T$ and $\|w-\vec v\|_{(L^\I\cH\cap\St)(0,\ti T)}\lec\ra$. If $\ti T=\I$, then 
\EQ{ \label{ux L6}
 \|w\|_{L^\I(0,\I;(1\oplus |\na|)L^{2^*}(\R^d))}+\|\vec u\xt\|_{L^\I(0,\I;(1\oplus |\na|)L^{2^*}(\R^d))} \lec \ra.} 

{\rm (II)} There exists $\{\vec u^\te\}_{0\le \te\le 1}$ in $\Sol([0,T])$ or $\Sol([0,T))$, which is $C^1$ in $\te$, such that  $\vec u^0=\vec u$, $\vec u^1=\vec u\dt$, $\vec u^\te=\vec u$ on $B_{+t}(0,T)$, $\|\vec u^\te(0)\|_{\Ra_B^{\ti T}}\lec\ra$, and $\vec u^\te$ satisfies {\rm (I)} for the fixed $u\dt$ and some $\te$-dependent $u\xt$ and $v$ for all $\te\in[0,1]$. 

{\rm (III)} Although such $u\dt$ and $u\xt$ in {\rm (I)} are not unique, we can define a $C^1$ map $A:\vec u(0)\mapsto\vec u\dt(0)\in\cH$ locally around each $\vec u(0)$ and fixed $B$. Moreover, it satisfies 
\EQ{ \label{Lip A}
 \|A(\fy^0)-A(\fy^1)-(\fy^0-\fy^1)\|_{\cH} \lec \|\fy^0-\fy^1\|_{\cHR{B^\cmpl}}.}
\end{lem}
Note that no energy bound is required on $u$, while the condition in $\Ra_B^{\ti T}$ can be satisfied either by localization as in \eqref{loc Ra} or by dispersion as in \eqref{disp Ra}, which is useful respectively for concentrating blow-up and for scattering solutions. 
\begin{proof}
The definition of $\Ra_B^{\ti T}$ yields a free solution $\vec v$ such that $\vec v(0)=\vec u(0)$ on $B^\cmpl$ and 
\EQ{
 \|\vec u(0)\|_{\Ra_B^{\ti T}} = \|\vec v\|_{L^\I_t(0,\ti T;\cHR{B_{+t}})\cap \St(0,\ti T)} \lec \|\vec v(0)\|_\cH \sim \|\vec u(0)\|_{\cHR{B^\cmpl}}.} 
Since $\ra<\e_S$, there is a unique $u\xt\in\Sol([0,\ti T])$ such that 
\EQ{
 \vec u\xt(0)=\vec v(0), \pq \|\vec u\xt-\vec v\|_{(L^\I\cH\cap\St)(0,\ti T)} \lec \|v\|_{\St(0,\ti T)}^{2^*-1} =\ra^{2^*-1},}
which, together with the above estimate on $v$, implies that 
\EQ{
 \|\vec u\xt\|_{L^\I_t(0,\ti T;\cHR{B_{+t}})\cap \St(0,\ti T)} \sim \ra.}
In addition, if $\ti T=\I$ then combining the above with \eqref{St L6} yields \eqref{ux L6} for $u\xt$.  
The propagation speed of (CW) implies that $\vec u\xt=\vec u$ on $(B_{+t})^\cmpl=(B^\cmpl)_{-t}$. 
If $\ti T<\I$, then let $w$ be the solution of  
\EQ{ \label{eq w}
 (\p_t^2-\De)w=f'(u\xt)- f'(u\xt-w),}
with $\vec w(\ti T)=(1-X_{B_{+\ti T}})u\xt(\ti T)$, where $X_{B_{+\ti T}}$ is the extension operator for $B_{+\ti T}$. Then 
\EQ{
 \|\vec w(\ti T)-\vec v(\ti T)\|_\cH \le \|\vec u\xt(\ti T)-\vec v(\ti T)\|_\cH+\|X_{B_{+\ti T}}\vec v(\ti T)\|_\cH \lec \ra,}
and so $\|U(t-\ti T)\vec w(\ti T)\|_{\St(0,\ti T)}\lec\ra$ by the Strichartz estimate. Also we have 
\EQ{
 \|\vec w-U(t-\ti T)\vec w(\ti T)\|_{\St} \lec \|\square w\|_{\St_s^*} \pt\lec \|w\|_{\St}(\|u\xt\|_{\St}+\|w\|_{\St})^{2^*-2} 
 \pr\lec \ra\|w\|_{\St}+\|w\|_{\St}^{2^*-1},
 }
thereby we obtain $w\in C([0,\ti T];\cH)$ satisfying 
$\|\vec w-\vec v\|_{L^\I(0,\ti T;\cH)} + \|w\|_{\St(0,\ti T)} \lec \ra$ 
as well as $\vec w=0$ on $B_{+t}([0,\ti T])$ by the finite propagation speed. 
Let $u\dt:=u-w\in C([0,T);\cH)$. Then we have, for $0<t<T$, $\vec u\dt=\vec u$ on $B_{+t}$ and  
\EQ{ \label{eq udt}
 \square u\dt=f'(u)-f'(u\xt)+f'(u\xt-w)=\CAS{f'(u)=f'(u\dt) &\ton{B_{+t}} \\ f'(u-w)=f'(u\dt) &\ton{(B_{+t})^\cmpl}},}
since $\vec w=0$ on $B_{+t}$ and $\vec u=\vec u\xt$ on $(B_{+t})^\cmpl$. 
$\vec u\dt-\vec u+\vec v=\vec v-\vec w$ has been already estimated above. 

To define $u^\te$ in (II), let $w^\te$ be the solution of \eqref{eq w} with $\vec w^\te(\ti T)=\te\vec w(\ti T)$ and let $u^\te:=u-w^\te$. Then obviously $w^0=0$, $w^1=w$, $\vec w=0$ on $B_{+t}(0,\ti T)$, and so $\square u^\te=f'(u^\te)$ in the same way as \eqref{eq udt}. The same estimate as above yields $\|w^\te-\te\vec w\|_{L^\I\cH\cap\St(0,\ti T)}\lec\ra$, and hence $\|\vec u^\te(0)-(1-\te)\vec u(0)-\te \vec u\dt(0)\|_\cH\lec\ra$, and $\|\vec u^\te(0)-\vec u(0)+\te v(0)\|_\cH\lec\ra$, which implies that 
\EQ{
 \|\vec u^\te(0)\|_{\Ra_B^{\ti T}} \le \|\vec u(0)\|_{\Ra_B^{\ti T}}+\te\|\vec v\|_{L^\I_t(0,\ti T;\cHR{B_{+t}})\cap \St(0,\ti T)}+C\ra \lec \ra.}
Hence $\vec u^\te$ satisfies (I) with the above constructed $u\dt$, the free solution $\vec v^\te(t):=\vec v(t)+U(t)[\vec u^\te(0)-\vec u(0)]$ and the associated nonlinear solution $u\xt$ (dependent on $\te$).

In the case $\ti T=\I$, we define a sequence $\vec w_n\in C([0,n];\cH)$ with $\ti T=n$ as above. Then the uniform bound allows us to take a weak limit along a subsequence to $w\in C([0,\I);\cH)\cap\St(0,\I)$ solving \eqref{eq w}, the estimates and $\vec w=0$ on $B_{+t}(0,\I)$. For (II), let $w^\te$ be the solution of the integral equation 
\EQ{ \label{Ieq w}
 \vec w=\int_\I^t U(t-t')(0,f'(u\xt+w)-f'(u\xt))dt'}
obtained by the iteration starting from $\te w$. Then $u^\te:=u-w^\te$ satisfies the desired properties, which is seen by the same argument as above. 

To define the map $A:\vec u(0)\mapsto\vec u\dt(0)$ in (III), we perturb $\vec u(0)$ around some fixed $\vec u^0(0)\in\cH$ satisfying the assumption. Let $\vec v^0$ be the free solution chosen as above for $\vec u^0$. For $\vec u(0)\in\cH$ close to $\vec u^0(0)$, we have 
\EQ{
 \|\vec u(0)\|_{\Ra_B^{\ti T}}\le\|\vec u^0(0)\|_{\Ra_B^{\ti T}}+C\|\vec u(0)-\vec u^0(0)\|_\cH<\e_S.}
We choose the free solution $\vec v$ for $\vec u$ as a perturbation from $\vec v^0$, putting 
\EQ{
 \vec v(0):=\vec v^0(0)+X_{B^\cmpl}(\vec u(0)-\vec u^0(0)).} 
Then in the same way as above, if $\ti T<\I$, let $\vec u\xt\in\Sol([0,\ti T])$ with $\vec u\xt(0)=\vec v(0)$, let $w$ be the solution of \eqref{eq w}, and let $\vec u\dt=\vec u-\vec w$. By the Strichartz estimate, we see that the maps $\vec u(0)\mapsto \vec v(0)=\vec u\xt(0)\mapsto \vec u\xt(\ti T)\mapsto \vec w(\ti T)\mapsto \vec w(0)\mapsto\vec u\dt(0)$ are $C^1$, where $\vec u\xt$ and $\vec w$ are Lipschitz with respect to $\vec u(0)\in\tHR{B^\cmpl}$, leading to \eqref{Lip A}. 

In the case $\ti T=\I$, we also fix $\vec w^0\in C([0,\I);\cH)$ for $u^0$, and then let $w$ be the solution of \eqref{Ieq w} obtained by the iteration starting from $w^0$. Then $w\in C([0,\I);\cH)$ solves $\square w=f'(u\xt+w)-f'(u\xt)$ on $t>0$, $\vec w=0$ on $B_{+t}$, and 
\EQ{
 \|\vec w-\vec w^0\|_{L^\I\cH\cap\St(0,\I)}\lec\|\vec v(0)-\vec v^0(0)\|_\cH.} 
Moreover, the maps $\vec u(0)\mapsto\vec u\xt(0)\mapsto \vec w(0)\mapsto \vec u\dt(0)$ are $C^1$, where $\vec u\xt$ and $\vec w$ are Lipschitz with respect to $\vec u(0)\in\tHR{B^\cmpl}$, leading to \eqref{Lip A}. 
\end{proof}

The following lemma is crucial to show that the ground state component is small in the region where the solution is dispersive. 
\begin{lem}[Reverse Sobolev for the ground state] 
For any $d\ge 3$, there exists $C=C(d)>0$ such that for any $B\in\sB_d$, 
\EQ{ \label{rev Sob}
 \|\vec W\|_{\ti\cH(B)}+\|g_\pm\|_{\ti\cH(B)} \le C\|W\|_{L^{2^*}(B)},
 \pq \|g_\pm\|_{\cHR{B}} \le C\|\vec W\|_{\cHR{B}}.}
\end{lem}
\begin{rem}
This lemma obviously fails for any other stationary solutions, since they have indefinite sign. To see that, concentrate $B$ at any zero point.
\end{rem}
\begin{proof}
For the first estimate, we may assume that $\|W\|_{L^{2^*}(B)}<1$, since the left hand side is uniformly bounded. Let $R:=\|W\|_{L^{2^*}(B)}^{-2/(d-2)}>1$. 
Since $W\sim\LR{x}^{2-d}$ and $|\na W|\sim |x|\LR{x}^{-d}$, we have 
\EQ{
 \|\na W\|_{L^2(B)}^2 \pt\le \int_{B\cap\{|x|<R\}}|\na W|^2dx + \int_{|x|>R}|\na W|^2 dx 
 \pr\lec \int_{B}R^2|W|^{2^*}dx + R^{2-d} \sim R^{2-d} = \|W\|_{L^{2^*}(B)}^{2}.}
The estimate on $g_\pm$ follows from that $\ro/W\in L^\I\cap W^{1,d}(\R^d)$. Indeed, let $\chi:=\ro/W$ and let $\fy=\vec W$ on $B$. Then we have $g_\pm=(2k)^{-1/2}(1,\pm k)\chi\fy_1$ on $B$, and  
\EQ{
 \|(1,\pm k)\chi\fy_1\|_{\cH}
 \le \|\chi\|_{L^\I}\|\na\fy_1\|_2+(\|\na\chi\|_{L^d}+\|\chi_1\|_{L^d})\|\fy_1\|_{L^{2^*}} \lec \|\fy\|_\cH.}
Hence $\|g_\pm\|_{\cHR{B}}\lec\|\vec W\|_{\cHR{B}}$. The estimate in $\ti\cH(B)$ is similar. 
\end{proof}

\section{Center-stable manifold with large radiation}
Now we can extend the center-stable manifold by adding large radiation.
Fix $\ra_m>0$ such that $\ra_m\ll \de_m$. 
 Let $\M_3$ be the totality of $\vec u(0)\in\cH$ such that for the solution $u$ we can apply the detaching Lemma \ref{lem:detach} with 
\EQ{ \label{def M3}
 \ti T=\I, \pq \vec u\dt(0)\in\M_0, \pq \ra+\|\sT_{\vec u\dt(0)}\vec W\|_{\cHR{B^\cmpl}}<\ra_m.} 
Then $\M_0\subset\M_3$ by using the trivial case $B=\R^d$ and $u\dt=u$.  
 The invariance of $\M_3$ for $\T,\cS$ is inherited from $\M_0$, which is also clear from the definition. 

For each point $\fy\in\M_3$, Lemma \ref{lem:detach} gives a neighborhood $O\ni\fy$ and a $C^1$ map $A:O\ni\vec u(0)\mapsto \vec u\dt(0)\in\cH$ such that $A(\fy)\in\M_0$. Reducing $O$ if necessary, we may assume that $A(O)$ is within the domain of $M_+$ defined in \eqref{def M+}, and that the last condition of \eqref{def M3} holds all over $O$.

Then we have $\M_3\cap O=(M_+\circ A)^{-1}(0)$. Indeed $\supset$ is clear from the definition of $\M_3$ and $M_+$. If $M_+(A(\psi))>0$, then the solution $u\dt$ starting from $A(\psi)$ blows up in some finite $T>0$, so does the solution $u$ starting from $\psi$, because of \eqref{v ux ud}, 
\EQ{ \label{far in L6}
 \dist_{L^{2^*}}(u(t),\Static(W)_1) \pt\ge \dist_{L^{2^*}}(u\dt(t),\Static(W)_1)-\|u(t)-u\dt(t)\|_{L^{2^*}} \pr\gg \de_m-O(\ra) \sim \de_m,}
for $t$ close to $T$. On the other hand, if $\vec u(0)\in\M_3$ with $\vec u\in\Sol([0,T))$, then
\EQ{ \label{close in L6}
 \dist_{L^{2^*}}(u(t),\Static(W)_1) \pt\le \dist_{L^{2^*}}(u\dt(t),\Static(W)_1)+\|u(t)-u\dt(t)\|_{L^{2^*}} \pr\lec \de_m+\ra \sim \de_m}
for $t$ close to $T$. Hence $M_+(A(\psi))>0$ implies that $\psi\not\in\M_3$. 
If $M_+(A(\psi))<0$, then \eqref{v ux ud} with Strichartz implies that the solution $u$ starting from $\psi$ also scatters to $0$, contradicting \eqref{close in L6}, and so $\psi\not\in\M_3$. 

In order to conclude that $\M_3$ is a $C^1$ manifold of codimension $1$, it now suffices to show that the $C^1$ functional $M_+\circ A$ does not degenerate on its zero set $\M_3$. Indeed, if $M_+(A(\psi))=0$ then $\p_h M_+(A(\psi+h\sT_{A(\psi)} g_+))\sim 1$, because by the Lipschitz property of $A$ \eqref{Lip A} and $m_+$ \eqref{tan of La}, and the last condition of \eqref{def M3} together with \eqref{rev Sob}, we have 
\EQ{
 \pt A(\psi+h\sT_{A(\psi)} g_+)=A(\psi)+h[\sT_{A(\psi)} g_+ + O(\ra_m;\cH)],
 \pr M_+(A(\psi+h\sT_{A(\psi)} g_+))=h+O(h\ra_m).}

By Lemma \ref{lem:detach}(II), we can connect each $\fy\in\M_3$ with some $\psi\in\M_0$ by a $C^1$ curve, which is included in an enlarged $\M_3$ for which the last bound in \eqref{def M3} is replaced with $C\ra_m$ for some constant $C>1$. Including those curves connecting $\M_3$ to $\M_0$, we obtain a slightly bigger  manifold $\ti\M_3$, which is $C^1$ and connected with codimension $1$.  

Let $\M_4$ be the maximal evolution of $\ti\M_3$ (in the same way as we define $\M_1$ from $\M_0$). Then $\M_4$ is a connected $C^1$ manifold of codimension $1$, which is invariant by the flow, $\T$ and $\cS$, and $\M_4\supset\M_1\cup\ti\M_3$. Every solution $\vec u\in\Sol([0,T))$ on $\M_4$ satisfies 
\EQ{
 \limsup_{t\nearrow T}\dist_{(1\oplus|\na|)L^{2^*}}(\vec u(t),\Static(W))\lec\de_m.}
Around each point on $\M_4$, there is a small open ball which is split into two open sets by $\M_4$, such that all the solutions starting from one of them blow up in finite time, near which time
\EQ{
 \dist_{L^{2^*}}(u(t),\Static(W)_1)\gg\de_m,}
and all those starting from the other scatter to $0$ as $t\to\I$. 

On the other hand, if $\vec u\in\Sol([0,\I))$ scatters to $\Static(W)$, namely
\EQ{ \label{scat2W}
 \exists\fy\in\cH,\pq d_W(\vec u(t)-U(t)\fy)\to 0\pq(t\to\I),}
then $\vec u(0)\in\M_4$. This is because Lemma \ref{lem:est Ra} implies that we can detach the free radiation $U(t)\fy$ at some large $t=T$, so that $\vec u\dt(t)=\sT_{\vec u(t)}\vec W+O(\ra)$ in $\cH$ for all $t\ge T$ with $\ra\ll\ra_m\ll\de_m$. The last condition of \eqref{def M3} is ensured by \eqref{est Ra}.

Finally, let $\M_5$ be the Lorentz extension of $\M_4$, defined in the same way as for $\M_2$ from $\M_1$. Note that all the solutions on $\M_4$ have the space-time maximal regions as in \eqref{bcurv on M} in the case of blow-up, since the remainder $u\xt$ is globally small in the Strichartz norms. Thus we obtain a connected $C^1$ manifold $\M_5\supset \Soliton(W)\cup\M_4\cup\M_2$ with codimension $1$ in $\cH$, which is invariant by the flow, $\T$, $\cS$ and the Lorentz transform. If $\vec u\in\Sol([0,\I))$ satisfies \eqref{scat2soliton}, then the scaling invariance and dispersion of the free wave $v$ implies that 
\EQ{
 P(\vec u)=\lim_{t\to\I}P(\vec W_{0,0,p(t)})+P(\vec v)
 = \lim_{t\to\I}p(t)E(W)+P(\vec v).}
Hence $p(t)$ converges to some $p_*\in\R^d$, and then $\vec W_{\la(t),c(t),p(t)}-\vec W_{\la(t),c(t),p_*}\to 0$ in $\cH$. Take a Lorentz transform which maps $\vec W_{0,0,p_*}$ to $\vec W$, and apply it to $\vec u$. Then we obtain another global solution satisfying \eqref{scat2soliton} with $p(t)\equiv 0$, namely scattering to $\Static(W)$, and so belonging to $\M_4$. Hence $u$ is on $\M_5$. 

Since each Lorentz transform defines a local $C^1$ diffeo around each solution, there is a neighborhood of $\M_5$ transformed from a neighborhood of $\M_4$, such that all solutions starting off the manifold within the neighborhood either scatter to $0$ as $t\to\I$ or blow up in $t>0$ away from a bigger neighborhood. 
Thus we obtain Theorem \ref{main thm3}. 

\section{One-pass theorem with large radiation}
In this section, we derive one-pass theorems which allow arbitrarily large radiation. For $\fy\in\cH$, we define the ``radiative distance" $\dR$ to the ground states for any $\fy\in\cH$ by
\EQ{ \label{def dRa}
 \dR(\fy):=\inf_{B\in\sB_d,\ \psi\in\pm\Static(W)}\|\fy-\psi\|_{\cHR{B}}+\|\fy\|_{\Ra_B^\I}+\|\psi\|_{\cHR{B^\cmpl}}.  }
Obviously $\dR:\cH\to[0,\I)$ is Lipschitz continuous, and $\dR(\T^c\cS^\s\fy)=\dR(\fy)$. It is not invariant for the time inversion $\fy\mapsto\fy^\da$. 
Taking $B=\R^d$ yields 
\EQ{
 \dR(\fy) \le \dist_W(\fy)\sim d_W(\fy).}
The embeddings \eqref{St L6} and $\cH\subset(1\oplus |\na|)L^{2^*}$ imply
\EQ{
 \dist_{(1\oplus|\na|)L^{2^*}}(\fy,\pm\Static(W)) \lec \dR(\fy).}
By Lemma \ref{lem:est Ra}, we immediately obtain
\begin{lem} \label{lem:est dR}
Under the same assumption as in Lemma \ref{lem:est Ra} with $I=[0,\I)$, 
\EQ{
 \lim_{t\to\I}\dR(\vec u(t)) \lec \ra^{d/2-1}+\ra.}
In particular, if $\vec u\in\Sol([0,\I))$ scatters to the ground states \eqref{scat2W}, then 
\EQ{
 \dR(\vec u(t))  + t^{-1}(e^{-\ti\s(\vec u(t))}+|\ti c(\vec u(t))|) \to 0 \pq(t\to\I).}
\end{lem}
\begin{proof}
Combine \eqref{est Ra} with
$\dR(\vec u(t)) \lec \|\vec F(t)\|_{\Ra_{B(t)}^\I}+\|\psi(t)\|_{\cHR{B(t)^\cmpl}}+\|R(t)\|_{\cH}$. 
\end{proof}

Smallness of the radiative distance enables the detaching.
\begin{lem} \label{lem:ud2dR}
Let $\vec u\in\Sol([0,T))$ satisfy $\dR(\vec u(0))=\ra<\e_S$. Then there exist $B\in\sB_d$, a free solution $v$, $\vec u\xt\in\Sol([0,\I))$ and $\vec u\dt\in\Sol([0,T))$, satisfying $\vec v(0)=\vec u(0)$ on $B^\cmpl$, $\vec u\xt=\vec u$ on $(B_{+t})^\cmpl$, $\vec u\dt=\vec u$ on $B_{+t}$, and 
\EQ{ \label{v ux ud dR}
 \pt \|\vec v\|_{L^\I_t(0,\I;\cHR{B_{+t}})\cap\St(0,\I)}+\|\vec u-\vec u\dt-\vec v\|_{L^\I\cH\cap\St(0,T)}\lec \ra,
 \pr \|\vec u\xt-\vec v\|_{L^\I\cH\cap\St(0,\I)}\lec\ra^{2^*-1}, 
 \pq \|\vec u\dt(t)\|_{L^\I_t(0,T;\cHR{B_{+t}^\cmpl})} \lec \ra.}
Moreover, for $0\le t<T$ we have 
\EQ{ \label{dR ud}
  \dR(\vec u(t))\lec \ra+d_W(\vec u\dt(t)).} 
\end{lem}
\begin{proof}
By definition of $\dR$, there exist $B\in\sB_d$ and $\psi\in\pm\Static(W)$ such that 
\EQ{ \label{def B psi}
 \ra\gec \|\vec u(0)-\psi\|_{\cHR{B}}+\|\vec u(0)\|_{\Ra_B^\I}+\|\psi\|_{\cHR{B^\cmpl}}<\e_S.}
Applying Lemma \ref{lem:detach} with $\ti T=\I$ yields $v$, $u\xt$, $u\dt$ and \eqref{v ux ud dR}, where the last inequality follows from the others 
\EQ{
 \|\vec u\dt(t)\|_{\cHR{B_{+t}^\cmpl}}
  \le \|\vec u\xt(t)-\vec v(t)\|_{\cHR{B_{+t}^\cmpl}}+\|\vec u\dt(t)-\vec u(t)+\vec v(t)\|_{\cH} \lec \ra.}
Let $\psi(t):[0,T)\to\pm\Static(W)$ such that $d_W(\vec u\dt(t))\sim\|\vec u\dt(t)-\psi(t)\|_\cH$. Then 
\EQ{ 
 \pt\|\vec u(s)-\psi\|_{\cHR{B_{+s}}}
 \le \|\vec u\dt(s)-\psi\|_{\cH} \sim d_W(\vec u\dt(s)),
 \pr \|\psi\|_{\cHR{(B_{+s})^\cmpl}}
 \le \|\vec u\dt(s)-\psi\|_{\cH}+\|u\dt(s)\|_{\cHR{(B_{+s})^\cmpl}}\lec d_W(\vec u\dt(s))+\ra,
 \pr\|\vec u(s)\|_{\Ra_{B_{+s}}^\I} 
   \le \|U(t-s)\vec u\xt(s)\|_{L^\I_t(s,\I;\cHR{B_{+t}})\cap\St(s,\I)}
 \prq\lec \|\vec v\|_{L^\I_t(s,\I;\cHR{B_{+t}})\cap\St(s,\I)}+\|\vec v(s)-\vec u\xt(s)\|_\cH \lec \ra.}
Gathering these three estimates, we obtain \eqref{dR ud}. 
\end{proof}

We are now ready to prove the first one-pass theorem in the radiative distance. 
\begin{thm} \label{thm:OPR1}
There exist constants $C_*>1>\ra_*>0$ such that if $\vec u\in\Sol([0,T])$ satisfies 
\EQ{
 \max(\dR(\vec u(0)),\dR(\vec u(T)))=:\ra \le \ra_* \ll \de_*,}  
then there is $B\in\sB_d$ such that the conclusion of Lemma \ref{lem:detach}(I) holds and 
\EQ{
 \|\dR(\vec u)\|_{L^\I_t(0,T)}+\|d_W(\vec u\dt)\|_{L^\I_t(0,T)}\le C_*\ra.} 

On the other hand, if $\vec u\in\Sol([0,T))$ satisfies 
\EQ{
 \dR(\vec u(0))=:\ra\le\ra_*, \pq \dR(\vec u(t_0))>C_*\ra}
at some $t_0\in(0,T)$, then there is $t_1\in[t_0,T)$ such that 
\EQ{
 \dR(\vec u(t_1))=\inf_{t_1\le t<T}\dR(\vec u(t))\ge \max(\ra,\ra_*/C_*).} 
Moreover, if $K_{B_{+t_1}}(u(t_1))<0$ then $T<\I$ and 
\EQ{
 \liminf_{t\nearrow T}\dist_{L^{2^*}}(u(t),\Static(W)_1)\gec \de_*,} 
otherwise $T=\I$ and $u$ scatters to $0$ as $t\to\I$. 
\end{thm}
\begin{proof}
We will reduce it to \cite{CNW-nonrad} by the detaching Lemma \ref{lem:detach}. 
Choose $\ra_*<\e_S$ and let $B$, $v$, $u\xt$ and $u\dt$ be as in the above lemma, with $\psi\in\pm\Static(W)$ satisfying \eqref{def B psi}. Combining it with \eqref{v ux ud dR} yields
\EQ{
 \|\vec u\dt(0)-\psi\|_{\cH}
 \pt\le \|\vec u(0)-\psi\|_{\cH(B)}+\|\psi\|_{\cH(B^\cmpl)}+\|\vec u\dt(0)\|_{\cH(B^\cmpl)} \lec \ra,}
which implies $E(\vec u\dt)=E(W)+O(\ra^2)$ since $E'=0$ on any static solution. 

The same argument at $t=T$ yields $\ti B\in\sB_d$ and $\ti\psi\in\pm\Static(W)$ such that 
\EQ{
 \pt \|\vec u(T)-\ti\psi\|_{\cHR{\ti B}}+\|\vec u(T)\|_{\Ra_{\ti B}^\I}+\|\ti \psi\|_{\cHR{{\ti B}^\cmpl}}+\|u(T)\|_{L^{2^*}({\ti B}^\cmpl)} \lec \ra.}
Let $\hat B:=B_{+T}$. Using the reverse Sobolev \eqref{rev Sob}, we obtain 
\EQ{
 \|\na\ti\psi_1\|_{L^2(\ti B\setminus \hat B)}
 \lec \|\ti\psi_1\|_{L^{2^*}(\ti B\setminus \hat B)}
 \lec \|\ti\psi-\vec u(T)\|_{\cHR{\ti B}}+\|u(T)\|_{L^{2^*}({\hat B}^\cmpl)}
 \lec \ra,}
and, combining it with \eqref{v ux ud dR} and \eqref{def B psi}, 
\EQ{ \label{ud in}
 \pn\|\vec u\dt(T)-\ti\psi\|_{\ti\cH(\ti B)} 
 \pt\le \|\vec u(T)-\ti\psi\|_{\ti\cH(\ti B\cap\hat B)} + \|\vec u\dt(T)\|_{\ti\cH(\ti B\setminus\hat B)}+\|\ti\psi\|_{\ti\cH(\ti B\setminus \hat B)} 
 \pr\lec \|\vec u(T)-\ti\psi\|_{\cHR{\ti B}}+\|\vec u\dt(T)\|_{\cHR{{\hat B}^\cmpl}}+ \ra \lec \ra.}
Expanding the energy for $\fy:=\vec u\dt(T)-\ti\psi$ yields
\EQ{
 E(\vec u\dt(T))=E(W) \pt+ \frac{1}{2}[\LR{L_{\ti\psi} \fy_1|\fy_1}+\|\fy_2\|_2^2] + o(\|\fy_1\|_{2^*}^2),}
where $L_{\ti\psi}:=-\De-f''(\ti \psi_1)$. Since $\|u\dt(T)\|_{L^{2^*}({\ti B}^\cmpl\cup{\hat B}^\cmpl)}
 + \|\ti\psi_1\|_{L^{2^*}({\ti B}^\cmpl\cup{\hat B}^\cmpl)} \lec \ra$ and \eqref{ud in}, we have $\|\fy_1\|_{L^{2^*}(\R^d)}\lec\ra$ and so 
\EQ{
 E(W)+O(\ra^2) = E(\vec u\dt(T))=E(W) \pt+ \frac{1}{2}\|\fy\|_\cH^2 + o(\ra^2).}
Thus we obtain 
\EQ{
 d_W(\vec u\dt(T))\lec \|\vec u\dt(T)-\ti\psi\|_\cH  \lec \ra.}
Choose $\ra_*\ll\e_*$ of \cite[Theorem 5.1]{CNW-nonrad}, so that we can apply that one-pass theorem to $u\dt$, both from $t=0$ forward in time, and from $t=T$ backward in time. Specifically, there are $\e,\de>0$ such that 
\EQ{
 \pt \ra\sim\e\sim\de,\pq \e\le\e_*,\pq \sqrt{2}\e<\de\le\de_*,
 \pr E(\vec u\dt)\le E(W)+\e^2, \pq \max(d_W(\vec u\dt(0)),d_W(\vec u\dt(T)))<\de}
so that we can deduce from the theorem that 
$\max_{0\le t\le T}d_W(\vec u\dt(t)) < \de$, and then, by the above lemma, 
$\dR(\vec u)\lec\de\sim\ra$ for $0<t<T$. 

Next, if $\dR(\vec u(t_0))\ge C_*\ra$ at some $t_0$, with $C_*>1$ large enough, then the above lemma implies that $d_W(\vec u\dt(t_0))>\de$ and so, by the classification of the dynamics after ejection in \cite{CNW-nonrad}, we conclude that $d_W(\vec u\dt(t))\gec\de_*$ after some $t_1\in(t_0,T)$, and then $u\dt$ either blows up in finite time, or scatters to $0$, so does $u$ by \eqref{v ux ud dR}. The blow up occurs away from the ground states in the sense of \eqref{far in L6}. Moreover, this dichotomy is determined by $\sign(K(u\dt(t_1)))$. Choosing $\ra_*$ smaller if necessary, and using \eqref{v ux ud dR}, we have
\EQ{
 \pm\de_* \sim K(u\dt(t_1))=K_{B_{+t_1}}(u(t_1))+O(\ra),}
which implies $\sign K_{B_{+t_1}}(u(t_1))=\sign K(u\dt(t_1))$. 
\end{proof}

In particular, we can characterize the manifold with large radiation, constructed in the previous section, by using the radiative distance. 
\begin{cor}
Let $\vec u\in\Sol([0,T))$. If $\vec u(0)\in\M_3$ then $\sup_{0<t<T}\dR(\vec u(t))\lec\de_m$. Conversely, if $\sup_{0<t<T}\dR(\vec u(t))<\min(\ra_m,\ra_*/C_*)$ then $\vec u(0)\in\M_3$. 
\end{cor}
\begin{proof}
Let $\vec u(0)\in\M_3$, then we can apply the detaching Lemma \ref{lem:detach} with \eqref{def M3}, so
\EQ{
 \dR(\vec u(0))\le \|\vec u\dt(0)-\psi\|_{\cHR{B}}+\|\vec u(0)\|_{\Ra_B^\I}+\|\psi\|_{\cHR{B^\cmpl}}
 \lec \de_m+\ra_m \lec \de_m,}
where $\psi:=\sT_{\vec u\dt(0)}\vec W$. Since $\vec u\dt(0)\in\M_0$, using \eqref{dR ud} we obtain $\dR(\vec u(t))\lec d_W(\vec u\dt(t))+\de_m\lec\de_m$. 

If $\sup_{0<t<T}\dR(\vec u(t))<\min(\ra_*,\ra_m/C_*)$, then the above theorem implies that 
\EQ{
 \dR(\vec u(0))+\sup_{0<t<T}d_W(\vec u\dt(t))<\ra_m\ll\de_m.} 
Hence $\vec u\dt(0)\in\M_0$ and the definition of $\dR$ implies that $\vec u(0)\in\M_3$. 
\end{proof}
We can choose those distance parameters such as $\ra_m=\de_m/C$ and $\ra_*=\ra_m/C$ with some large absolute constant $C>1$, provided that $\de_m$ is chosen much smaller than the energy parameter $\e_*>0$ of the one-pass theorem in \cite{CNW-nonrad}. 

The above one-pass theorem does not preclude oscillation between $\ra<\dR<C_*\ra$. In the case of $d_W$ in \cite{CNW-nonrad}, it was possible to exclude such oscillations completely thanks to the convexity in time of $d_W^2$ near $\Static(W)$, which is not inherited by $\dR$. 
However, we can make an exact version of the above one-pass theorem by the flow. 
\begin{thm}[One-pass theorem with large radiation] \label{thm:OPR2}
There exist constants $C_*>1>\ra_*>0$, and an open set $X(\ra)\subset\cH$ for each $\ra\in(0,\ra_*]$ satisfying: 
\begin{enumerate}
\item $X(\ra)$ is increasing, i.e. $\ra_1<\ra_2\implies X(\ra_1)\subset X(\ra_2)$. 
\item Its boundary is in $\ra \le \dR \le C_*\ra$, namely
\EQ{
 \dR(\fy)<\ra \implies \fy\in X(\ra)\implies \dR(\fy)<C_*\ra.} 
\item \label{no-ret} No solution can return to it, namely for any $\vec u\in\Sol([0,T))$ 
\EQ{
 \vec u(0)\in X(\ra),\ \exists t_0\in(0,T),\ \vec u(t_0)\not\in X(\ra)
 \implies \forall t\in[t_0,T),\ \vec u(t)\not\in X(\ra).}
Moreover, such $u$ either blows up in finite time or scatters to $0$ in $t>t_0$.\end{enumerate}
\end{thm}
\begin{proof}
Let $\ra_*>0$ and $C_*>1$ be as in the previous theorem, though we will modify them at the end of proof. For $0<\ra \le \ra_*/C_*^2$, let $X(\ra)$ be the totality of the initial data $\vec u(0)$ of any solution $\vec u\in\Sol((T_-,T_+))$ with $T_-<0<T_+$ satisfying 
\EQ{
 \inf_{T_-<t\le 0}\dR(\vec u(t))<\ra, \pq \inf_{0\le t<T_+}\dR(\vec u(t))<C_*\ra.}
Since $\dR$ is continuous in $\cH$, the local wellposedness implies that $X(\ra)$ is open. Since the left quantity is non-increasing while the right quantity is non-decreasing along the flow, the no-return property \eqref{no-ret} is obvious. At the exit time $t_0$ we have 
\EQ{
 \dR(\vec u(t_0))=\inf_{t_0\le t<T_+}\dR(\vec u(t))=C_*\ra,}
and the previous theorem implies that such a solution $u$ either scatters to $0$ or blows up in $t>t_0$. It also implies that $\dR(\vec u(0))\le C_*^2\ra$. By definition we have $\dR\ge\ra$ on $X(\ra)^\cmpl$. 
Hence replacing $\ra_*$ with $\ra_*/C_*^2$ and then $C_*$ with $C_*^2$, we obtain the desired properties of $X(\ra)$. 
\end{proof}

\appendix
\section{Concentration blowup in the interior of blowup region} \label{sect:bup int}
Here we observe that type-II blow-up is {\it not always} on the dynamical boundary between the scattering to $0$ and blow-up. More precisely, we have
\begin{prop}
Let $0<T<\I$, $\vec u^0\in\Sol([0,T))\cap L^\I([0,T);\cH)$. Then for any $\de>0$, there is $\vec u^1\in\Sol([0,T))$ with the following property: $\vec u^1(t)-\vec u^0(t)$ has a strong limit as $t\nearrow T$, and for any $t\in[0,T)$ and any $\psi\in\cH$ with $\|\psi\|_\cH<\de$, the solution starting from $\vec u^1(t)+\psi$ blows up in positive finite time. 
\end{prop}
In other words, for any blow-up with bounded energy norm, there is another solution with the same blow-up profile, whose orbit is in the interior of the blow-up set of initial data, with arbitrarily large distance from the exterior. 
\begin{proof}
Fix $R \ge 1+T$ such that $\|\vec u^0(0)\|_{\cH(|x|>R)} \ll 1$. 
Let $u^2$ be the solution for the initial data $\vec u^2(0)=\Ga(x/R)\vec u^0(0)$, where $\Ga$ is a smooth radial function on $\R^d$ satisfying $\Ga(x)=1$ for $|x|\le 3$ and $\Ga(x)=0$ for $|x|\ge 4$. Then the finite speed of propagation implies that for $0<t<T$ and as long as $u^2$ exists, 
\begin{enumerate}
\item $\vec u^2(t)=\vec u^0(t)$ on $|x|<3R-t$, \label{u2 inner}
\item $\|\vec u^2(t)\|_{\cH(|x|>R+t)}\ll 1$, \label{u2 xsmall}
\item $\supp\vec u^2(t)\subset\{|x|<4R+t\}$.
\end{enumerate}
Since the regions for \eqref{u2 inner} and for \eqref{u2 xsmall} cover $[0,T)\times\R^d$, we deduce that $u^2$ extends beyond $t<T$. Moreover, both $u^0$ and $u^2$ extend to $|x|>R+t$ for all $t>0$ by the smallness in the exterior cone. Hence $\vec u^2(t)-\vec u^0(t)$ has a strong limit in $\cH$ as $t\nearrow T$. 

Now fix $\de>0$. Since $u^2$ is bounded in $\cH$ for $0<t<T$, 
\EQ{
 M:=\sup\{E_{|x|<5R}(\vec u^2(t)+\psi)\mid t\in[0,T),\ \|\psi\|_\cH<\de\}}
is finite. Then we can find a strong radial solution $u^3$ such that 
\begin{enumerate}
\item $\supp\vec u^3(t)\subset\{|x|>6R-t\}$. \label{u3 supp}
\item $\sup\{E_{|x|>5R}(\vec u^3(t)+\psi)\mid t\in[0,T),\ \|\psi\|_\cH<\de\}<-M-1$. \label{u3 negE}
\end{enumerate}
Indeed, it is easy to satisfy \eqref{u3 supp} and \eqref{u3 negE} at $t=0$ by using a very flat radial smooth function, since for any $\fy,\psi\in\cH$ and any $0<\e\ll 1$, 
\EQ{
 E_{|x|>5R}(\fy+\psi) \le (1+\e)E_{|x|>5R}(\fy)+C_\e(\|\psi\|_\cH^2+\|\psi_1\|_{2^*}^{2^*}).}
\eqref{u3 supp} is preserved for $t>0$ by the finite speed of propagation. For such initial data, the solution may blow up in finite time, but we can delay the blow-up time as much as we like by the rescaling $\vec S^\s$ with $\s\to-\I$, which makes both \eqref{u3 supp} and \eqref{u3 negE} easier. This yields $u^3\in \Sol([0,2T])$ with the above properties. 

Now let $u^1$ be the strong solution for the initial data 
\EQ{
 \vec u^1(0)=\vec u^2(0)+\vec u^3(0).}
Then the finite propagation property together with the disjoint supports of $u^2$ and $u^3$ implies that $\vec u^1=\vec u^2$ for $|x|<6R-t$, $\vec u^1=\vec u^3$ for $|x|>4R+t$, so $\vec u^1=\vec u^2+\vec u^3$ for $0<t<T$, and $\vec u^1(t)-\vec u^0(t)$ has a strong limit in $\cH$ as $t\nearrow T$. Moreover, for any $t\in[0,T)$ and any $\psi\in\cH$ satisfying $\|\psi\|<\de$ we have 
\EQ{
 E(\vec u^1(t)+\psi)
 =E_{|x|<5R}(\vec u^2(t)+\psi) + E_{|x|>5R}(\vec u^3(t)+\psi)<-1,}
hence the solution starting from $\vec u^1(t)+\psi$ has to blow up in finite time because of the negative energy, see \cite{Lev,KM}. 
\end{proof}

\section{Table of Notation}
{\small
\begin{longtable}{l|l|l}
 \hline 
 $\diff{X^\pa}$ & $=X^1-X^0$ & \eqref{def diff} \\
 $\vec u$ & $=(u,\dot u)$ vector in the phase space & \eqref{def H} \\
 $\fy^\da$ & $=(\fy_1,-\fy_2)$ time inversion & \eqref{def da} \\
 $\LR{\cdot|\cdot}$ & $L^2$ inner product & \eqref{def L2prod} \\ 
 $\sB_d$ & Borel sets in $\R^d$ & \\
\hline
 (CW) & the critical wave equation & \eqref{eqCW} \\ 
 $\cH$, $\cH_\perp$  & energy space, its subspace & \eqref{def H},\eqref{def Horth}\\
 $\Sol(I)$ & solutions of (CW) on $I$ & \eqref{def Sol} \\ 
 $E(\vec u),P(\vec u)$ & total energy and momentum & \eqref{def En},\eqref{def P} \\
 $E_B(\fy),K_B(\fy)$ & restricted energy functionals & \eqref{def EB} \\ 
 $U(t)$ & free propagator & \eqref{def U} \\ 
 $\T^c,\cS^\s,S_a^\s$ & invariant translation and scaling & \eqref{def TS} \\ 
 $\St,\St_*^*,q_s,q_m$ & Strichartz norms and exponents & \eqref{def St} \\
\hline
 $W$, $\Static(W)$ & ground states & \eqref{def W},\eqref{def Static}\\ 
 $\Soliton(W)$ & ground solitons & \eqref{def Soliton} \\
 $\dist_W$, $d_W$ & distances to the ground states & \eqref{def distW},\eqref{def dW} \\ 
 $L_+$, $\cL$ & linearized operators around $W$ & \eqref{def L+},\eqref{def cL} \\ 
 $N(v),\U{N}(v)$ & higher order terms & \eqref{def N},\eqref{eq vta} \\ 
 $\ro$, $k$, $P_\perp$ & ground state of $L_+$ & \eqref{def ro}\\ 
 $g_\pm$, $\La_\pm$ & (un)stable modes of $J\cL$ & \eqref{def la+-},\eqref{def La+-} \\ 
\hline
 $v,\la,\ga,\la_\pm$ & components of $u$ around $W$ & \eqref{decop u},\eqref{def la+-} \\
 $(\al,\mu)$ & parameters to define the orthogonality & \eqref{orth} \\
 $(\ti\s,\ti c)$, $\sT_\fy$, $\ti\la$  & local coordinates by the orthogonality & \eqref{def sT}, Lemma \ref{lem:orth} \\
 $Z=(Z_1,Z_2)$ & modulation operator in the equation & \eqref{eq vta} \\ 
 $\ta$ & rescaled time variable & \eqref{def ta} \\ 
 $\|\fy\|_E$, $\nu(\ta)$ & linearized energy norms & \eqref{def E},\eqref{def nu} \\
 $\B_\de$, $\B_\de^+$, $\B_\de'$ & small balls for different components & \eqref{def BN}, \eqref{def B+'} \\ 
 $\cN_\de$, $\cN_{\de_1,\de_2}$ & neighborhoods of $\Static(W)$ & \eqref{def BN},\eqref{def N2} \\ 
 $\Phi_{\s,c}$, $\Psi_{\s,c}$ & local coordinates around $\Static(W)$ & \eqref{def PhiPsi} \\ 
 $\M_0\sim\M_5$ & local manifold and its extensions & \eqref{def M0},\eqref{def M1},\eqref{def M2},\eqref{def M3} \\ 
 $m_+,M_+$ & functions defining the local manifold & Theorem \ref{thm:cmd0},\eqref{def M+} \\
 \hline 
 $a_W,b_W$ & positive constants & \eqref{der s almu},\eqref{der y almu} \\ 
 $\e_S$ & small Strichartz norm for scattering & \eqref{def eS} \\
 $\de_\Phi$, $\de_m$ & small distances from $\Static(W)$ & Lemma \ref{lem:orth}, Theorem \ref{thm:cmd0} \\
 $\ig_I$ & smallness in the ignition lemma & Lemma \ref{lem:inst} \\

 $\y_l$ & $\ta$-length for uniform Strichartz bound & Lemma \ref{lem:LUS} \\ 
 $\ra_m,\ra_*$ & smallness in radiative distance & \eqref{def M3}, Theorems \ref{thm:OPR1},\ref{thm:OPR2} \\
 $\e_*,\de_*$ & smallness in the one-pass theorem & \cite[Theorem 5.1]{CNW-nonrad} \\
 $\ka(\de)$ & variational bound on $K$ & \cite[Lemma 4.1]{CNW-nonrad} \\ 

 \hline 
 $B_{+a},B_{-a}$ & fattened and thinned sets by radius $a$ & \eqref{def B+-},\eqref{def B+-I} \\
 $\tHR{B},\ti\cH(B)$ & restrictions of $\cH$ to $B$ & \eqref{def cHR},\eqref{def ticH} \\
 $X_B$ & extension operator from $B$ to $\R^d$ & \eqref{def XB} \\
 $\dom(u)$, $t_\pm(\fy,x)$ & maximal space-time domain of solution & \eqref{def D} \\
 \hline 
 $\Ra_B^T$ & seminorm measuring radiation & \eqref{def Ra} \\
 $u\dt$,$u\xt$ & detached interior and exterior solutions & Lemma \ref{lem:detach} \\ 
 $\dR(\fy)$ & radiative distance to the ground states & \eqref{def dRa} \\
 \hline
\end{longtable}
}


\begin{thebibliography}{10}
\bibitem{Beceanu} M.~Beceanu,
{\it A centre-stable manifold for the energy-critical wave equation in $\R^3$ in the symmetric setting,}
preprint, arXiv:1212.2285. 
\bibitem{Bizon} P.~Bizo\'n, T.~Chmaj and Z.~Tabor, 
{\it On blowup for semilinear wave equations with a focusing nonlinearity,} 
Nonlinearity {\bf 17} (2004), no.~6, 2187--2201.
\bibitem{DHKS} R.~Donninger, M.~Huang, J.~Krieger and W.~Schlag,
{\it Exotic blowup solutions for the $u^5$ focusing wave equation in $\R^3$,} 
preprint, arXiv:1212.4718. 
\bibitem{DK} R.~Donninger and J.~Krieger, 
{\it Nonscattering solutions and blow up at infinity for the critical wave equation,}
preprint, arXiv: 1201.3258v1. 
\bibitem{DKM1} T.~Duyckaerts, C.~Kenig and F.~Merle, 
{\it Universality of blow-up profile for small radial type II blow-up solutions of energy-critical wave equation,} 
J. Eur. Math. Soc. {\bf 13} (2011), no.~3, 533--599.
\bibitem{DKM2} T.~Duyckaerts, C.~Kenig and F.~Merle,
{\it Universality of the blow-up profile for small type II blow-up solutions of energy-critical wave equation: the non-radial case,} 
preprint, arXiv:1003.0625, to appear in JEMS. 
\bibitem{DKM3} T.~Duyckaerts, C.~Kenig and F.~Merle, 
{\it Profiles of bounded radial solutions of the focusing, energy-critical wave equation,}  
preprint, arXiv:1201.4986, to appear in GAFA. 
\bibitem{DKM4} T.~Duyckaerts, C.~Kenig and F.~Merle, 
{\it Classification of radial solutions of the focusing, energy-critical wave equation,} 
preprint, arXiv:1204.0031. 
\bibitem{DM} T.~Duyckaerts and F.~Merle, 
{\it Dynamic of threshold solutions for energy-critical wave equation,} 
Int. Math. Res. Pap. (2008). 
\bibitem{HR} M.~Hillairet and P.~Rapha\"el, 
{\it Smooth type II blow up solutions to the four dimensional energy critical wave equation,}
preprint, arXiv:1010.1768. 
\bibitem{KM} C.~Kenig and F.~Merle, 
{\it Global well-posedness, scattering and blow-up for the energy-critical focusing non-linear wave equation,}  
Acta Math. {\bf 201} (2008),  no.~2, 147--212.
\bibitem{KN} J.~Krieger and J.~Nahas, 
{\it Instability of type II blow up for the quintic nonlinear wave equation on $\R^{3+1}$,}
preprint, arXiv:1212.4628. 
\bibitem{CNW-rad} J.~Krieger, K.~Nakanishi and W.~Schlag,
{\it Global dynamics away from the ground state for the energy-critical nonlinear wave equation,} 
to appear in Amer. J. Math.
\bibitem{CNW-nonrad} J.~Krieger, K.~Nakanishi and W.~Schlag, 
{\it Global dynamics of the nonradial energy-critical wave equation above the ground state energy,} 
DCDS {\bf 33} (2013), no.~6. 
\bibitem{partI} J.~Krieger, K.~Nakanishi and W.~Schlag, 
{\it Threshold phenomenon for the quintic wave equation in three dimensions,}
preprint, arXiv:1209.0347. 
\bibitem{KS} J.~Krieger and W.~Schlag, 
{\it On the focusing critical semi-linear wave equation,}  
Amer. J. Math. {\bf 129} (2007), no.~3, 843--913. 
\bibitem{KST} J.~Krieger, W.~Schlag and D.~Tataru, 
{\it Slow blow-up solutions for the $H^1(\R^3)$ critical focusing semilinear wave equation,}  
Duke Math. J. {\bf 147} (2009), no.~1, 1--53.
\bibitem{Lev} H.~A.~Levine, 
{\it Instability and nonexistence of global solutions to nonlinear wave equations of the form $Pu_{tt} = -A u + \mathcal{F}(u)$,} 
Trans. Amer. Math. Soc. {\bf 192} (1974), 1--21. 
\bibitem{NLKG-rad} K.~Nakanishi and W.~Schlag, 
{\it Global dynamics above the ground state energy for the focusing nonlinear Klein-Gordon equation,}  
Journal Diff. Eq. {\bf 250} (2011), 2299--2233.
\bibitem{NLS} K.~Nakanishi and W.~Schlag, 
{\it  Global dynamics above the ground state energy  for the cubic NLS equation in 3D,} 
Calc. Var. and PDE, {\bf 44} (2012), no.~1-2, 1--45. 
\bibitem{NLKG-nonrad} K.~Nakanishi and W.~Schlag, 
{\it Global dynamics above the ground state for the nonlinear Klein-Gordon equation without a radial assumption,}  
Arch. Rational Mech. Analysis, {\bf 203} (2012), no.~3, 809--851.
\bibitem{Inv} K.~Nakanishi and W.~Schlag, 
{\it Invariant Manifolds Around Soliton Manifolds for the Nonlinear Klein-Gordon Equation}, SIAM J. Math. Anal., {\bf 44} (2012), no.~2, 1175--1210.
\bibitem{PS} L. E. Payne and D. H. Sattinger, 
{\it Saddle points and instability of nonlinear hyperbolic equations,}
Israel J. Math. {\bf 22} (1975), no.~3-4, 273--303.
\bibitem{S} N. Szpak, 
{\it Relaxation to intermediate attractors in nonlinear wave equations},
Theor. Math. Physics, 127, 817 (2001).
\end{thebibliography}
\end{document}